\documentclass[10pt,a4paper,english,french]{amsart}
\usepackage{mathpazo}
\usepackage[T1]{fontenc}
\usepackage[latin9]{inputenc}
\usepackage{color}
\usepackage{amsthm}
\usepackage{graphicx}
\usepackage{amssymb}

\providecommand{\tabularnewline}{\\}

\numberwithin{equation}{section} %% Comment out for sequentially-numbered
\numberwithin{figure}{section} %% Comment out for sequentially-numbered
\theoremstyle{plain}
\theoremstyle{plain}
\newtheorem{thm}{Théorème}
  \theoremstyle{definition}
  \newtheorem{defn}[thm]{Définition}
 \theoremstyle{definition}
  \newtheorem{example}[thm]{Exemple}
  \theoremstyle{remark}
  \newtheorem{rem}[thm]{Remarque}
  \theoremstyle{plain}
  \newtheorem{lem}[thm]{Lemme}
  \theoremstyle{remark}
  \newtheorem{notation}[thm]{Notation}
  \theoremstyle{plain}
  \newtheorem{cor}[thm]{Corollaire}
  \theoremstyle{plain}
  \newtheorem{prop}[thm]{Proposition}

\usepackage{babel}
\addto\extrasfrench{\providecommand{\fg}{\ifdim\lastskip>\z@\unskip\fi~\frqq}}

\begin{document}

\title{Une vue panoramique sur l'analyse semi-classique}

\author{Olivier Lablée}

\maketitle
\tableofcontents{}

\section{Introduction}

Dans ce papier, on présente un rapide et partiel survol de l'analyse
semi-classique. Cette théorie mathématique est un carrefour entre
la géométrie et la théorie spectrale. Comme ce chapitre ne contient
que des résultats standards, on ne trouvera que quelques démonstrations,
mais on donnera des références précises. Ce chapitre ne prétend à
aucune originalité dans les résultats.

Pour plus de détails sur la géométrie symplectiques voir par exemple
les premiers chapitres des livres de M. Audin\textbf{ {[}8{]}} et
\textbf{{[}9{]}}, voir aussi les ouvrages de A.T. Fomenko \textbf{{[}62{]}},
\textbf{{[}63{]}}. On peut aussi consulter le très bon livre de D.
McDuff et D. Salomon \textbf{{[}96{]}}. Pour les généralités sur le
formalisme mathématique de la mécanique quantique on peut se reporter
à\textbf{ {[}76{]}}. Pour la partie théorie spectrale, voir par exemple
le livre de P. Lévy-Bruhl \textbf{{[}93{]}}, le livre de T. Kato \textbf{{[}85{]}},
ou encore la collection des Reed-Simon \textbf{{[}106{]}}. En ce qui
concerne la quantification voir par exemple le livre A. Catteneo,
B. Keller, C. Torrosian et A. Bruguières \textbf{{[}30{]}}. Enfin,
pour l'analyse semi-classique voir les références classiques : le
livre de D. Robert \textbf{{[}107{]}}, celui de M. Dimassi et J. Sjöstrand
\textbf{{[}52{]}}, le livre de A. Martinez \textbf{{[}94{]}} et le
polycopié de Y. Colin de Verdière \textbf{{[}107{]}}. Pour la partie
analyse microlocale, voir {[}\textbf{118}{]}, \textbf{{[}119{]}},
\textbf{{[}107{]}}, \textbf{{[}73{]}}. En ce qui concerne les systèmes
intégrables symplectique on pourra voir \textbf{{[}8{]}}, \textbf{{[}9{]}},
\textbf{{[}99{]}} et \textbf{{[}119{]}}. Le récent livre de S. Vu
Ngoc \textbf{{[}119{]}} propose un grand panorama très complet sur
les systèmes intégrables symplectique et semi-classique.

\section{La géométrie symplectique}

A la différence de la géométrie riemannienne, la géométrie symplectique
est une géométrie de mesure de surface, dédiée à la base pour la formulation
de la mécanique de Hamilton, elle joue aussi un rôle très important
à l'intérieur même des mathématiques, notamment en topologie. Pour
commencer on définira la notion de variété symplectique, on donnera
ensuite des exemples simples, comme par exemple le fait que pour n'importe
quelle variété différentiable $M$, on peut munir son fibré cotangent
$T^{*}M$ d'une structure symplectique. On verra ensuite les principales
caractéristiques de la géométrie symplectique: comme la structure
de Lie induite sur l'algèbre des fonctions $\mathcal{C}^{\infty}(M)$.
Puis on finira sur l'absence de géométrie locale, ce qui constitue
encore une différence majeure avec le cas riemannien.

\subsection{La mécanique de Hamilton}

La mécanique de Hamilton par rapport à la formulation de Lagrange
n'apporte rien de nouveau sur le contenu physique, mais elle offre
un cadre géométrique puissant, elle apporte une nouvelle façon de
voir la physique: une façon moderne et géométrique. Une des principales
caractéristique de la physique moderne, c'est la géométrie (relativité
générale, cordes...). La géométrie riemannienne est une généralisation
de l'ancienne géométrie euclidienne, elle est liée à la théorie de
la relativité générale et à la théorie des jauges. Il existe une autre
géométrie, encore plus liée à la physique:\textbf{ }la géométrie\textbf{
}symplectique. Moins connue que sa cousine riemannienne, elle est
pourtant très riche, elle formalise parfaitement la mécanique de Hamilton.
Dans la théorie d'Hamilton, les particules physiques sont décrites
par leurs positions et leurs vitesses; par exemple dans l'espace euclidien
$\mathbb{R}^{3}$ un point est caractérisé par un vecteur $\left(x_{1},x_{2},x_{3},\xi_{1},\xi_{2},\xi_{3}\right)\in\mathbb{R}^{6}.$
Pour un hamitonien $f\in\mathcal{C}^{\infty}(\mathbb{R}^{6},\mathbb{R})$,
la dynamique est alors donnée par les équations de Hamilton :\[
\left\{ \begin{array}{cc}
\dot{\xi_{j}}=-\frac{\partial f}{\partial x_{j}}\\
\\\dot{x_{j}}=\frac{\partial f}{\partial\xi_{j}}\mathbf{.}\end{array}\right.\]

\subsection{Les variétés symplectiques}
\begin{defn}
Une variété symplectique $\left(M,\omega\right)$ est une variété
différentiable de dimension $m$ muni d'une $2$-forme $\omega$ fermée
et non-dégénérée.
\end{defn}
Rappelons que non-dégénérée signifie que pour tout point $x$ de $M$,
la forme bilinéaire $\omega(x)$ est non-dégénérée sur l'espace vectoriel
$T_{x}^{*}M$. Ainsi comme pour tout point $x$ de $M$, la forme
bilinéaire $\omega$ est à la fois non-dégénérée et alternée, la dimension
de l'espace $T_{x}^{*}M$ doit être nécessairement pair; ainsi la
dimension de la variété $M$ est elle aussi paire : $m=2n$. 

Une variété symplectique $\left(M,\omega\right)$ est naturellement
munie d'une forme volume $\tau=\frac{1}{n!}\omega^{n}$, où $\omega^{n}=\omega\wedge\omega\wedge...\wedge\omega$,
on a donc à disposition une orientation et une mesure de Lebesgue
sur la variété $M$. 

Donnons quelques exemples standards de variétés symplectiques:
\begin{example}
\textbf{L'espace $\mathbb{R}^{2n}$} : c'est l'exemple type, en notant
par $\left(x_{1},x_{2},...,x_{2n}\right)$ la base canonique de l'espace
vectoriel $\mathbb{R}^{2n}$, et en notant par $dx_{j}:=\left(x_{j}\right)^{*},\; j\in\left\{ 1,...,2n\right\} $
la base duale de $\mathbb{R}^{2n}$; alors la 2-forme\[
\omega_{0}={\displaystyle \sum_{i,j=1}^{n}dx_{i}\wedge dx_{j}}\]
 munie la variété $\mathbb{R}^{2n}$ d'une structure symplectique.
Si en particulier $n=1$, $\omega_{0}$ est un déterminant, c'est
alors une mesure d'aire algébrique sur le plan.\end{example}
\begin{rem}
Il est bon de noter que la géométrie symplectique ne donne que des
notions d'aire, il n'y a pas de notions de longueurs, encore moins
de notions d'angles, en effet, tout vecteur est orthogonal à lui même!
\end{rem}
Donnons un exemple sur une variété compacte :
\begin{example}
\textbf{La sphère $\mathbb{S}^{2}$:} sur $\mathbb{S}^{2}$ on définit
pour tout $x\in\mathbb{S}^{2}$: \[
\omega_{x}(\eta,\xi):=\left\langle x,\eta\wedge\xi\right\rangle _{\mathbb{R}^{3}}\]
avec $(\eta,\xi)\in\left(T_{x}\mathbb{S}^{2}\right)^{2}\subset\mathbb{R}^{3}$
ce qui munit la sphère d'une structure de variété symplectique.
\end{example}
Plus généralement on a :
\begin{example}
\textbf{Sur une  surface }: il suffit de prendre :\[
\omega=\sqrt{|g|}\, dq\wedge dp\]
où $|g|=det(g_{i,j})$ dans la carte de coordonnées locales $(p,q)$.
\end{example}
Finissons maintenant un exemple fondamental :
\begin{example}
\textbf{Sur le fibré cotangent d'une variété} : Le fibré cotangent
d'une variété différentiable est naturellement muni d'une structure
symplectique. En effet, pour toute variété $M$ lisse de dimension
$n$, on peut munir de façon intrinsèque son fibré cotangent $T^{*}M$
d'une structure de variété symplectique $\left(T^{*}M,\omega\right)$
de dimension $2n$ définie par la différentielle extérieure\[
\omega=d\alpha\]
de la 1-forme de Liouville $\alpha$.
\end{example}

\subsection{Premiers résultats et applications}

Rappelons d'abord une notion importante: celle de symplectomorphisme. 
\begin{defn}
Soient $\left(M_{1},\omega_{1}\right)$ et $\left(M_{2},\omega_{2}\right)$
deux variétés symplectique de même dimension, un symplectomorphisme
de $M_{1}$ sur $M_{2}$ est un difféomorphisme $f:\; M_{1}\rightarrow M_{2}$
tels que :\[
f^{*}\omega_{2}=\omega_{1}.\]

\end{defn}
Sur les variétés symplectique, le premier résultat géométrique majeur
est le théorème de Darboux qui donne la {}``géométrie'' locale des
ces variétés. 
\begin{thm}
\textbf{(Darboux).} Toute variété symplectique $\left(M,\omega\right)$
de dimension $2n$ est localement symplectomorphe à $\left(\mathbb{R}^{2n},\omega_{0}\right)$.
\end{thm}
Ce qui signifie que pour tout point $x_{0}\in M$, il existe un ouvert
$U$ de $M$ contenant $x_{0}$ et il existe un système $\left(x_{1},...,x_{n},\zeta_{1},...,\zeta_{n}\right)$
de coordonnées locales%
\footnote{dites coordonnées canonique, ou encore coordonnées de Darboux.%
} tels que sur l'ouvert $U$ on ait l'expression : \[
\omega={\displaystyle \sum_{i,j=1}^{n}d\zeta_{i}\wedge dx_{j}.}\]
Le théorème de Darboux établit une différence majeure entre les géométries
riemannienne et symplectique, en effet: dans le premier cas, il y
a un invariant local: la courbure, alors que dans le second cas tout
est localement isomorphe à $\left(\mathbb{R}^{2n},\omega_{0}\right)$:

\subsection{Flot hamiltonien}

Comme la 2-forme $\omega$ est non-dégénérée, pour tout point $x_{0}$
de $M$, on peut avec la 2-forme $\omega$, identifier les deux espaces
vectoriels $T_{x_{0}}^{*}M$ et $T_{x_{0}}M$. Ainsi pour $f\in\mathcal{C}^{\infty}(M)$
par dualité il existe un unique vecteur $\chi_{f}\left(x_{0}\right)\in T_{x_{0}}M$
tel que pour tout $v_{x_{0}}\in T_{x_{0}}M$ on ait :\[
\omega(x_{0})\left(\chi_{f}(x_{0}),v_{x_{0}}\right)=-df(x_{0}).v_{x_{0}}.\]
Ensuite, pris fibre par fibre, nous avons l'existence et l'unicité
d'un champs de vecteur $\chi_{f}\in\Gamma(M)$ tel que pour tout champs
de vecteur $v\in\Gamma(M)$ on ait :\[
\omega\left(\chi_{f},v\right)=-df.v\]
c'est-à-dire : 

\[
i_{\chi_{f}}(\omega)=-df.\]
En coordonnés locales de Darboux on a alors l'écriture :\[
\chi_{f}={\displaystyle \sum_{j=1}^{n}\frac{\partial f}{\partial\xi_{j}}\left(\frac{\partial}{\partial x_{j}}\right)-\frac{\partial f}{\partial x_{j}}\left(\frac{\partial}{\partial\xi_{j}}\right).}\]
On note alors par $\varphi_{t}^{f}$ le flot associé au champs de
vecteur $\chi_{f}$ : $\varphi_{t}^{f}:\, m=m\mapsto\varphi_{t}^{f}(m)$.
Ce flot est donné comme étant la trajectoire associé au champs de
vecteur $\chi_{f}$ passant par $m$. C'est à dire que:\[
\left\{ \begin{array}{cc}
\frac{d}{dt}\left(\varphi_{t}^{f}(m)\right)=\chi_{f}\left(\varphi_{t}^{f}(m)\right)\\
\\\varphi_{0}^{f}(m)=m\end{array}\right.\]
où $\frac{d}{dt}\left(\varphi_{t}^{f}(m)\right)\in T_{\varphi_{t}^{f}(m)}M.$ 

En coordonnés de Darboux on a l'expression familière des équations
de Hamilton:\[
\left\{ \begin{array}{cc}
\dot{\xi_{j}}=-\frac{\partial f}{\partial x_{j}}\\
\\\dot{x_{j}}=\frac{\partial f}{\partial\xi_{j}}\mathbf{.}\end{array}\right.\]
Sur une variété riemannienne $\left(X,g\right)$ le flot géodésique
est donné par :\[
G_{t}\left(x,v\right)=(\gamma(t),\dot{\gamma}(t))\]
où $\gamma$ est la géodésique de $X$ telle que $\gamma(0)=x\in X$
et $\dot{\gamma}(t)=T_{x}X$. Le groupe $G_{t}$ à un paramètre est
un groupe de difféomorphismes de $T^{*}X$ (que l'on a identifié à
$TX$ via la métrique $g$) qui conserve le fibré unitaire $UX$.
Ce flot est exactement le flot hamiltonien sur $T^{*}X$ associé à
la fonction\[
f(x,\xi)=\frac{1}{2}\sum_{i,j}g^{i,j}(x)\xi_{i}\xi_{j}.\]

\begin{example}
\textbf{Exemple de calcul de flot:} L'oscillateur harmonique en dimension
2: si on considère la variété différentiable canonique $M=\mathbb{R}^{2}$,
son fibré cotangent $N=T^{*}M=\mathbb{R}^{4}$ a une structure de
variété symplectique. Notons par $H$ l'oscillateur harmonique en
dimension 2 :\[
H(x_{1},x_{2},\xi_{1},\xi_{2})=H_{1}(x_{1},x_{2},\xi_{1},\xi_{2})+H_{2}(x_{1},x_{2},\xi_{1},\xi_{2})=\alpha_{1}\frac{x_{1}^{2}+\xi_{1}^{2}}{2}+\alpha_{2}\frac{x_{2}^{2}+\xi_{2}^{2}}{2}\]
avec $\alpha_{1},\alpha_{2}$ deux réels strictements positifs. 

On va calculer le flot hamiltonien de point initial : \[
m_{0}=\left(\begin{array}{c}
x_{1,0}\\
x_{2,0}\\
\xi_{1,0}\\
\xi_{2,0}\end{array}\right)\in H_{1}^{-1}(E_{1})\cap H_{2}^{-1}(E_{2})\]
où $E_{1}>0$ et $E_{2}>0$.

Les équations de Hamilton sont alors :\[
\left(\begin{array}{c}
\dot{x_{1}}(t)\\
\dot{x_{2}}(t)\\
\dot{\xi_{1}}(t)\\
\dot{\xi_{2}}(t)\end{array}\right)=\left(\begin{array}{c}
\alpha_{1}\xi_{1}(t)\\
\alpha_{2}\xi_{2}(t)\\
-\alpha_{1}x_{1}(t)\\
-\alpha_{2}x_{2}(t)\end{array}\right).\]
En posant pour $j\in\left\{ 1,2\right\} $, $Z_{j}(t):=x_{j}(t)+i\xi_{j}(t)\in\mathbb{C}$,
on a immédiatement que $\dot{Z_{j}(t)}=-i\alpha_{j}Z_{j}(t)$ et donc
en intégrant cette équation différentielle linéaire d'ordre 1 on obtient:\[
Z_{j}(t)=Z_{j}(0)e^{-i\alpha_{j}t}\]
et\[
|Z_{j}(0)|^{2}=x_{j}^{2}(0)+\xi_{j}^{2}(0)=\frac{2E_{j}}{\alpha_{j}}\]
donc le flot hamiltonien est tracé sur le tore :\[
\mathbb{T}^{2}=\sqrt{\frac{2E_{1}}{\alpha_{1}}}\mathbb{S}^{1}\times\sqrt{\frac{2E_{2}}{\alpha_{2}}}\mathbb{S}^{1}.\]
En coordonnées angulaire le flot s'écrit alors linéairement:

\[
\varphi_{t}:\,\left(\begin{array}{c}
\theta_{1,0}\\
\theta_{2,0}\end{array}\right)\mapsto\left(\begin{array}{c}
t\alpha_{1}+\theta_{1,0}\\
t\alpha_{2}+\theta_{2,0}\end{array}\right).\]
\end{example}
\begin{rem}
Revenons maintenant a un fait important, qui justifie que la 2-forme
doit être fermée: la forme $\omega$ (ainsi que la forme volume associée)
est conservée par le flot hamiltonien :\[
\forall t\geq0\;\left(\varphi_{t}^{f}\right)^{*}\omega=\omega\]
pour le voir il suffit d'écrire la définition de la dérivée de Lie:
$\mathcal{L}_{\chi_{f}}(\omega)=\frac{d}{dt}\left(\varphi_{t}^{f}\right)^{*}$et
d'utiliser la formule de Cartan pour voir que $\mathcal{L}_{\chi_{f}}(\omega)=0$.
\end{rem}

\subsection{Crochet de Poisson}

A partir du champs $\chi$ on peut munir $\mathcal{C}^{\infty}(M)$
d'une structure d'algèbre de Lie : avoir une application à valeurs
dans $\mathcal{C}^{\infty}(M)$ , bilinéaire, alternée et vérifiant
l'identité de Jacobi. Pour toutes fonctions $(f,g)\in\left(\mathcal{C}^{\infty}(M)\right)^{2}$
on définit le crochet de Poisson:\[
\left\{ .,.\right\} :\left\{ \begin{array}{cc}
\mathcal{C}^{\infty}(M)\times\mathcal{C}^{\infty}(M)\rightarrow\mathcal{C}^{\infty}(M)\\
\\(f,g)\mapsto\left\{ f,g\right\} :=\omega(\chi_{f},\chi_{g}).\end{array}\right.\]
En fait l'application:\[
\chi:\;\mathcal{C}^{\infty}(M),\,\left\{ .,.\right\} \rightarrow\Gamma(M),\,\left[.,.\right]\]
est un homéomorphisme d'algèbre de Lie. En coordonnées de Darboux,
le crochet de Poisson de deux fonctions $f$ et $g$ s'écrit simplement
:\[
\left\{ f,g\right\} ={\displaystyle \sum_{j=1}^{n}\frac{\partial f}{\partial\xi_{j}}\frac{\partial g}{\partial x_{j}}-\frac{\partial f}{\partial x_{j}}\frac{\partial g}{\partial\xi_{j}}.}\]
Soient $f$ et $g$ deux fonctions de $\mathcal{C}^{\infty}(M)$,
et notons simplement par $\varphi_{t}$ le flot hamiltonien associé
à $f$, alors pour tout point $m\in M$ nous avons que : \[
\frac{d}{dt}\left(g\circ\varphi_{t}(m)\right)=\left\{ g,f\right\} \circ\varphi_{t}(m).\]
Ainsi la fonction $g$ est constante le long des trajectoires du flot
hamiltonien associé à $f$, si et seulement si $\left\{ g,f\right\} =0.$

\subsection{Un peu de topologie symplectique }

Finissons cette partie sur un résultat étonnant. Comme on la vue,
la géométrie symplectique est isochore : elle conserve le volume,
de plus par l'absence de courbure, elle est moins rigide que la géométrie
riemannienne, et on peut espérer que dans une telle géométrie on peut
{}``faire passer un chameau par le chat d'une aiguille'', plus précisément
on peut se demander si on peut plonger de manière symplectique une
boule dans un cylindre de rayon plus petit. Cette question est restée
ouverte jusqu'en 1985 où M. Gromov à répondu par la négation \textbf{{[}74{]}}.

En notant $B^{2n}(r):=B_{\mathbb{R}^{2n}}(0,r)$ et $Z^{2n}(R):=B_{\mathbb{R}^{2n}}(0,r)\times\mathbb{R}^{2n-2}$
on a le :
\begin{thm}
\textbf{(Gromov).} Si il existe un plongement symplectique (un plongement
qui conserve $\omega_{0}$) de $B^{2n}(r)$ dans $Z^{2n}(R)$, alors
$r\leq R$.
\end{thm}

\section{Mécanique quantique et théorie spectrale }

\subsection{La révolution de la physique }

La physique jusqu'à la fin du XIX siècle était constituée par deux
entitées: tout d'abord les corpuscules, c'est-à-dire les points matériels,
qui constituent la matière, la seconde entité est les ondes, qui constituent
les vibrations et les rayonnements. Le mouvement des corpuscules est
décrit par des trajectoires déterministes dans l'espace, on connaît
à tout instant la position et la vitesse d'une corpuscule. Les ondes
sont quand à elles non localisées, elle amené des phénomènes d'interférences.
Alors que les physiciens penser encore pouvoir tout décrire avec ces
deux entitées, certaines expériences, comme celle des fentes de Young,
de la photo de Wilson (1911) et même l'effet photo-électrique qui
datait de 1887 ne pouvait s'expliquer avec les deux entitées de base.
On peut alors situer la naissance de la mécanique quantique au moment
où est a débuté l'interrogation des physiciens au sujet de l'interprétation
des ces fameuses expériences. Au début de XX siècle, ce furent les
idées révolutionnaires de Bohr, Einstein, Heinsenberg, et Schrödinger
qui fournirent une théorie respectable. Construite pour expliquer
des phénomènes de rayonnement, cette théorie débouche sur beaucoup
d'autres thèmes de la physique. Un des grands succès de la théorie
quantique est que celle-ci s'attaque directement à la structure fondamentale
de la matière, en expliquant notamment les structures moléculaires,
atomiques et les propriétés des électrons. La mécanique quantique
est à la fois un bouleversement intellectuel, culturel et philosophique.
En effet, c'est une toute nouvelle façon de penser, opposée à l'intuition
immédiate, et qui est nécessaire pour comprendre le monde sous un
aspect quantique. Par sa puissance analytique et prédictive, elle
a permis d'ouvrir de nouvelles voies dans la recherche scientifique
et dans l'évolution de la technologie.

La mécanique quantique décrit la réalité physique avec des principes
et des postulats. Afin de mieux comprendre l'origine du cadre mathématique
de cette théorie, citons quelque postulats (pour plus de détails voir
par exemple \textbf{{[}97{]}}). Dans l'espace euclidien $\mathbb{R}^{n}$
la description quantique d'un point à l'instant $t$ se fait avec
une fonction d'onde, c'est-à-dire un vecteur $\varphi(t)\in L^{2}(\mathbb{R}^{n})$
que l'on interprète de la manière suivante : pour toute partie $\Omega$
de $\mathbb{R}^{n}$ le réel \[
\int_{\Omega}|\varphi(t)|^{2}dx_{1}...dx_{n}\]
est la probabilité de présence de la particule dans le domaine $\Omega$
à l'instant $t$; bien sur ceci impose la normalisation : \[
\int_{\mathbb{R}^{n}}|\varphi(t)|^{2}dx_{1}...dx_{n}=1.\]
Un second principe donne la dynamique (l'analogue quantique des équations
de Hamilton) : lorsque la particule est soumise à un champs de forces
dérivant d'un potentiel $V$, la fonction d'onde associé vérifie l'équation
de Schrödinger 

\[
ih\frac{d\varphi(t)}{dt}=H\varphi(t)\]
où $H=-\frac{\hbar^{2}}{2}\Delta+V.$ L'équation de Schrödinger tient
sa justification physique des ses conséquences. 

Un autre principe, encore bon a mentionner est celui concernant les
observables: a toute grandeur physique $\mathbf{a}$ on lui associe
un opérateur linéaire (à domaine) auto-adjoint $A$ agissant sur l'espaces
des fonction d'ondes de sorte que le réel \[
\int_{\Omega}A\varphi(t)\overline{\varphi(t)}dx_{1}...dx_{n}\]
représente la valeur moyenne des résultats de la mesure de la grandeur
$\mathbf{a}$.

\subsection{Rappels de théorie spectrale hilbertienne}

La théorie des opérateurs linéaires à domaine et l'étude de leurs
réduction, ce qu'on nomme théorie spectrale constitue les fondements
mathématiques de la mécanique quantique. En 1932, J. Von Neumman donne
la définition abstraite des espaces de Hilbert et il montre que les
points de vue de Heinsenberg et de Schrödinger sont équivalent, en
même temps il développe la théorie de réduction des opérateurs à domaine.
Le théorème spectral est l'un des résultat les plus profond de l'analyse
moderne, et est fondamental en mécanique quantique. Grâce à ce théorème
on peut introduire la notion de mesure spectrale associé à un état,
cette mesure conduit à une loi de probabilité sur $\mathbb{R}$, qui
amene à l'interprétation probabiliste de la mécanique quantique. Le
théorème spectral permet aussi des opérations sur les observables,
comme la composition, qui est similaire a ce qu'on sait déjà faire
en mécanique classique. En mécanique quantique, les observables sont
des opérateurs auto-adjoints à domaine, et d'un point de vue technique,
la difficulté de définir des fonctions d'opérateurs non bornée est
levée en grande partie par le théorème spectral, plus précisément
par son corollaire: le calcul fonctionnel. Un autre fait remarquable
en mécanique quantique est que les opérateurs ne commute pas, ceci
conduit aux fameuses relations d'incertitude de Heinsenberg.

On va rappeler quelques résultats de théorie spectrale avec cette
fois des démonstrations : le lemme central au théorème spectral concernant
les opérateurs non bornés est le suivant:
\begin{lem}
Soit $(X,\mathcal{F},\mu)$ un espace mesuré, que l'on notera plus
simplement $(X,\mu)$ et $F$ une fonction réelle finie $\mu$ presque
partout sur $X$. Alors l'opérateur de multiplication par $F$ définit
par :\begin{eqnarray*}
M_{F}:\; D(M_{F})=\left\{ \varphi\in L^{2}(X,\mu)/F\varphi\in L^{2}(X,\mu)\right\} \rightarrow L^{2}(X,\mu)\end{eqnarray*}
\begin{eqnarray*}
\varphi\mapsto F\varphi\end{eqnarray*}
est un opérateur à domaine dense et auto-adjoint. En outre si $F$
est bornée, l'opérateur $M_{F}$ est continue sur $L^{2}(X,\mu)$
avec $\left|\left|\left|M_{F}\right|\right|\right|\leq\left\Vert F\right\Vert _{\infty}$
.\end{lem}
\begin{proof}
Vérifions dans un premier temps que $D(M_{F})$ est dense dans $L^{2}(X,\mu)$:
Pour cela soit $\varphi\in L^{2}(X,\mu)$ et considérons la suite
$(\varphi_{n})_{n}$ de $L^{2}(X,\mu)$ définie par $\varphi_{n}=\varphi\chi_{(\left|F\right|\leq n)},$
$\chi$ désignant la fonction indicatrice. Compte tenu que pour tout
$n\in\mathbb{N}\,,\,|F\varphi_{n}|\leq n\varphi$, on a que $(\varphi_{n})_{n}$
est une suite de $D(M_{f})$. Ensuite comme $F$ est finie presque
partout il est clair que $(\varphi_{n})_{n}$ converge simplement
presque partout sur $X$ vers $\varphi$. Enfin comme pour tout $n\in\mathbb{N}\,,\,|\varphi_{n}|\leq\varphi\in L^{2}(X,\mu)$
on à grâce au théorème de convergence dominé de Lebesgue que $(\varphi_{n})_{n}$
converge dans $L^{2}(X,\mu)$ vers $\varphi$. 

Vérifions que l'opérateur $M_{F}$ est férmé : considérons donc une
suite $(\varphi_{n})_{n}\in(D(M_{F}))^{\mathbb{N}}$ telle que $\varphi_{n}$
converge dans $L^{2}(X,\mu)$ vers un certain $\varphi\in L^{2}(X,\mu)$
et que $M_{F}(\varphi_{n})$ converge dans $L^{2}(X,\mu)$ vers un
certain $\psi\in L^{2}(X,\mu)$. Avec la réciproque du théorème de
convergence dominé de Lebesgue on à d'une part que, quitte à extraire
une sous suite, $\varphi_{n}$ converge simplement presque partout
sur $X$ vers $\varphi$, donc en particulier $F\varphi_{n}$ converge
simplement presque partout sur $X$ vers $F\varphi$, ainsi $\psi=F\varphi$
presque partout. Donc $\varphi\in D(M_{F})$ et $M_{F}(\varphi)=\psi$. 

Ensuite montrons que $M_{F}$ est auto-adjoint: comme $F$ est à valeurs
réelles $M_{F}$ est trivialement symétrique sur $D(M_{F})$, maintenant
pour montrer que $M_{F}$ est auto-adjoint montrons que $\ker(M_{F}^{*}\pm iI)=\left\{ 0\right\} $,
prenons donc $\varphi\in\ker(M_{F}^{*}-iI)$, on a d'une part que
$M_{F}^{*}(\varphi)=i\varphi$, et d'autre part comme:\[
\forall\psi\in D(M_{F}),\;<M_{F}^{*}(\varphi),\psi>_{L^{2}}=<\varphi,M_{F}(\psi)>_{L^{2}}\]
on a que\[
\forall\psi\in D(M_{F}),\;\int_{X}i\varphi\overline{\psi}\, d\mu=\int_{X}\varphi\overline{F\psi}\, d\mu\]
c'est à dire:\[
\forall\psi\in D(M_{F}),\;\int_{X}\varphi(i-F)\overline{\psi}\, d\mu=0.\]
Donc comme $D(M_{F})$ est dense dans $L^{2}(X,\mu)$ et $F$ est
à valeurs réelles, on a que $\varphi=0$. De même on a que $\ker(M_{F}^{*}+iI)=\left\{ 0\right\} $.
Donc $M_{F}$ est bien auto-adjoint. 

Ensuite le cas où $F$ bornée est trivial.
\end{proof}
Citons un théorème issu de la théorie spectrale sur les $\mathbb{C}^{*}$algèbres
\textbf{{[}6{]}}, \textbf{{[}7{]}} : 
\begin{thm}
\textbf{(Théorème spectral des opérateurs normaux bornés).} Soit $N$
un opérateur borné normal sur un hilbert $\mathcal{H}$ séparable.
Il existe un espace mesuré fini $(Y$,$\mu)$ et une fonction bornée
$F$ sur $Y$, tels que $N$ soit unitairement équivalent à la multiplication
par $F$ dans $L^{2}(X,\mu)$ au sens du lemme .
\end{thm}
A partir de ce théorème, on va pouvoir en donner une version pour
les opérateurs auto-adjoint non bornés. Commençons par des notations
:
\begin{notation}
On va considérer un opérateur à domaine $(A,D(A))$ auto-adjoint sur
un hilbert $\mathcal{H},<.>_{\mathcal{H}}$ séparable. Comme $\sigma(A)\subset\mathbb{R}$,
on a que $\pm i$ ne sont pas dans le spectre de $A$, on notera alors
$R_{\pm i}=(\pm iI+A)^{-1}$ les résolvantes en ces points.
\end{notation}
Du précédent théorème et avec les notations ci-dessus on a :
\begin{cor}
Avec les notations précédentes, il existe un espace mesuré fini $(Y,\mu)$,
un opérateur unitaire $U:\mathcal{H}\rightarrow L^{2}(Y,\mu)$ et
une fonction bornée $F$ non nulle $\mu$ presque partout sur $Y$,
tels que avec les notations du lemme:\[
UR_{i}U^{-1}=M_{F}.\]
\end{cor}
\begin{proof}
Vérifions dans un premier temps que les opérateurs $R_{\pm i}$ sont
normaux : comme $Im(\pm iI+A)=\mathcal{H}$, on en déduit donc que
: $\forall(u,v)\in\mathcal{H}^{2}$, $\exists!(z,w)\in D(A)^{2}$
tel que $v=(-iI+A)(z)$ et $u=(iI+A)(w)$, donc on a :\[
\left\langle v,R_{i}u\right\rangle _{\mathcal{H}}=\left\langle (-iI+A)(z),w\right\rangle _{\mathcal{H}}\]
\[
=\left\langle z,(-iI+A)^{*}w\right\rangle _{\mathcal{H}}=\left\langle z,(iI+A)w\right\rangle _{\mathcal{H}}\]
\[
=\left\langle R_{-i}(v),u\right\rangle _{\mathcal{H}}\]
ce qui prouve que $R_{\pm i}^{*}=R_{\mp i}$. D'autre part d'après
l'équation résolvante, les résolvantes $R_{i}$ et $R_{-i}$ commutent,
ainsi $R_{i}$ et $R_{-i}$ sont des opérateurs bornés normaux. On
utilise ensuite le théorème : le seul point non contenu dans le théorème
est la non nullité de la fonction $F$: par définition $R_{i}$ est
injectif, donc par conjugaison unitaire $M_{F}$ doit l'être aussi,
c'est à dire que $F$ soit non nulle $\mu$ presque partout.
\end{proof}
Du cas particulier de la résolvante on a le théorème spectral sous
sa forme multiplicative pour les opérateurs auto-adjoints à domaine:
\begin{cor}
\textbf{(Théorème spectral multiplicatif).} Toujours avec les notations
précédentes, il existe un espace mesuré fini $(Y,\mu)$, un opérateur
unitaire $U:\,\mathcal{H}\rightarrow L^{2}(Y,\mu)$ et une fonction
réelle $f$ finie $\mu$ presque partout sur $Y$, tels que avec les
notations du lemme $v\in D(A)\Leftrightarrow U(v)\in D(M_{f})$ et
puis \[
UAU^{-1}=M_{f}\; sur\; U(D(A))\]
\end{cor}
\begin{proof}
Posons $f=\frac{1}{F}-i$ qui est donc finie $\mu$ presque partout
sur $Y$ car $F$ est elle même non nulle $\mu$ presque partout sur
$Y$. Notons bien que par construction $F(f+i)=1$ et $UR_{i}=M_{F}U$. 

Montrons que $v\in D(A)\Leftrightarrow U(v)\in D(M_{f})$: Soit $v\in D(A)$
alors $\exists!u\in\mathcal{H}$ tel que $v=R_{i}u$. Donc $Uv=M_{F}Uu$.
En multipliant cette égalité par $f$ on a : $fUv=(1-iF)Uu$. Comme
$Uu\in L^{2}(Y,\mu)$ et que la fonction $F$ est bornée on en déduit
que $fUv\in L^{2}(Y,\mu)$, ie: $Uv\in D(M_{f})$. Réciproquement
si $Uv\in D(M_{f})$ alors d'une part $(f+i)Uv\in L^{2}(Y,\mu)$,
d'autre part comme $U$ est unitaire entre $\mathcal{H}$ et $L^{2}(Y,\mu)$,
$\exists!u\in\mathcal{H}$ tel que $(f+i)Uv=Uu$. Enfin en multipliant
cette égalité par $F$ on obtient : $Uv=FUu$, donc $v=U^{-1}FUu=R_{i}u$
donc en particulier $v\in D(A)$. 

Montrons maintenant $UAU^{-1}=M_{f}\; sur\; U(D(A))$: notons bien
que $\forall u\in\mathcal{H}$ $FUu=UR_{i}u$ et donc $Uu=\frac{1}{F}UR_{i}u$.
Prenons $v\in D(A)$ $\exists!u\in\mathcal{H}$ tel que $v=R_{i}u$;
avec l'équation résolvante on obtient que $Av=u-iv$, donc $UAv=(\frac{1}{F}-i)Uv$,
soit encore $UAv=fUv$ ce qui montre l'égalité.

Reste à vérifier que la fonction $f$ est à valeurs réelles : si on
suppose que la partie imaginaire $\textrm{Im}(f)$ de $f$ est strictement
positive sur un sous ensemble $\Omega$ de $Y$ tels que $\mu(\Omega)>0$,
compte tenu que $(Y,\mu)$ est un mesuré finie $\chi_{\Omega}\in D(M_{f})$.
Maintenant on sait avec le lemme précédent que $M_{f}$ est auto-adjoint,
donc $\left\langle \chi_{\Omega},\, f\chi_{\Omega}\right\rangle _{L^{2}}\in\mathbb{R}$,
ce qui est absurde par définition de $\Omega$, donc au final $f$
est à valeurs réelles $\mu$ presque partout. Ce qui montre le corollaire.$\square$
\end{proof}
A partir de ce corollaire on en déduit le calcul fonctionnel borné
standard:
\begin{defn}
Soit $h\in L^{\infty}(\mathbb{R})$, avec les notations du corollaire
précédent, on définit l'opérateur borné $h(A)$ sur $\mathcal{H}$
par : \[
h(A)=U^{-1}M_{hof}U.\]

\end{defn}
On montre facilement le théorème suivant qui servira fortement dans
la suite :
\begin{thm}
\textbf{(Calcul fonctionnel borné)}. Soit $A,D(A)$ un opérateur auto-adjoint
sur $\mathcal{H}$. L'application $\phi$ définit par :\[
\phi:\left\{ \begin{array}{cc}
L^{\infty}(\mathbb{R})\rightarrow L_{c}(\mathcal{H})\\
\\h\mapsto h(A)\end{array}\right.\]
vérifie les propriétés suivantes:

\textbf{(i)} $\phi$ est un $\star$-homomorphisme d'algèbre normées
continu, et on a :\[
\forall h\in L^{\infty}(\mathbb{R}),\;\left|\left|\left|h(A)\right|\right|\right|\leq\left\Vert h\right\Vert _{\infty}\]
\textbf{(ii)} Si $(h_{n})_{n}\in(L^{\infty}(\mathbb{R}))^{\mathbb{N}}$
telle que $(x\mapsto h_{n}(x))_{n}$ converge simplement vers $x\mapsto x$
et $\forall(x,n)\in\mathbb{R}\times\mathbb{N}$ on ait $\left|h_{n}(x)\right|\leq\left|x\right|$
alors :\[
\forall u\in D(T),\; h_{n}(A)u\rightarrow Au\; dans\,\mathcal{H}\]
\textbf{(iii)} Si $(\lambda,u)\in\sigma_{p}(A)\times(\mathcal{H}-{0})$
tels que $Au=\lambda u$, alors si $h\in L^{\infty}(\mathbb{R}),$
$h(\lambda)\in\sigma_{p}(h(T))$ et $h(A)u=h(\lambda)u$.
\end{thm}

\subsection{Généralités sur la dynamique quantique}

La dynamique constitue un aspect essentiel de la mécanique quantique,
elle détermine au cours du temps les états quantiques, et par conséquent
l'espace de Hilbert des états ou des opérateurs agissant sur cet espace.
En considérant une grandeur physique mesurable par un observateur,
les postulats de la mécanique quantique indiquent qu'il peut être
associé à cette grandeur physique un opérateur auto-adjoint agissant
sur l'espace des états, et que le résultat de la mesure donnera: soit
la valeur propre de l'opérateur considéré si l'état quantique est
unique (cas pur), soit la valeur propre pondérée par la probabilité
d'existence d'un état quantique. La mesure de l'observable peut changer
au cours du temps, est-ce l'état quantique qui va évoluer au cours
du temps? Ou est-ce l'opérateur? Ou encore les deux en même temps?
Ces différents points de vue conduisent à des descriptions différentes
de la dynamique quantique. Le point de vue de Schrödinger est que
l'espace des états du Hilbert évoluent au cours du temps tandis que
les opérateurs sont invariant temporellement. Mathématiquement, à
partir d'un opérateur auto-adjoint on peut définir la dynamique quantique
via le calcul fonctionnel borné comme étant un groupe unitaire à un
paramètre (ici le temps) fortement continu, en effet (voir par exemple
\textbf{{[}85{]}} ou \textbf{{[}93{]}}):
\begin{thm}
Soit $A,D(A)$ un opérateur auto-adjoint sur un hilbert $\mathcal{H}$
alors la famille d'opérateur bornée:\begin{eqnarray*}
U(t)=\left\{ e^{itA}\right\} _{t\in\mathbb{R}}\end{eqnarray*}
est un groupe unitaire fortement continu, de générateur $\left(iA,D(A)\right)$.
\end{thm}
L'évolution d'un système physique $S$ au cours du temps peut donc
être représentée mathématiquement par un groupe unitaire de générateur
$(iA,D(A))$ dans un hilbert $\mathcal{H}$ associè à $S$, plus précisément
si l'état initial à l'instant 0 est représenté par le vecteur $\psi_{0}\in D(A)$,
l'état à l'instant $t$ est représenté par le vecteur:

\begin{eqnarray*}
\psi(t)=U(t)\psi_{0}\end{eqnarray*}
avec 

\begin{eqnarray*}
U(t)=e^{itA}\in L_{c}(\mathcal{H}).\end{eqnarray*}
L'opérateur $H=-hA$ est appelé hamiltonien du système $S$ et représente
l'observable énergie totale de $S$, on a donc que pour tout $\psi_{0}\in D(A)$:\begin{eqnarray*}
U(t)\psi_{0}=e^{-i\frac{t}{h}H}\psi_{0}.\end{eqnarray*}

\section{Spectre du laplacien et de l'opérateur de Schrödinger}

On va faire quelques rappel sur la théorie spectrale du laplacien
et de l'opérateur de Schrödinger. Pour plus de détails on pourra consulter
mon article\textbf{ {[}90{]}} qui donne un panorama partiel et historique
sur l'étude spectrale du laplacien et de l'opérateur de Schrödinger
sur des variétés riemanniennes.

Dans un système physique constitué d'une particule se déplaçant dans
une partie ouverte $X$ de $\mathbb{R}^{n}$, l'espace de Hilbert
associé est $L^{2}(X)$, et, si la particule n'est soumise à aucune
force, l'hamiltonien est :\[
H_{0}=-\frac{\hbar^{2}}{2m}\Delta\]
où $\Delta={\displaystyle \sum_{j=1}^{n}}\frac{\partial^{2}}{\partial x_{j}^{2}}$
est le laplacien de $\mathbb{R}^{n}$, $m$ la masse de la particule,
et $\hbar$ la constante de Planck. Si au contraire la particule est
soumise à un champ de force dérivant d'un potentiel réel $V$, l'hamiltonien
est alors :

\[
H=H_{0}+V\]
$V$ désignant l'opérateur de multiplication par la fonction $V$. 

En géométrie riemannienne, l'opérateur de Laplace-Beltrami%
\footnote{On utilise ici la convention de signe des analystes pour l'opérateur
de Laplace-Beltrami. Dans la convention des géomètres $\Delta_{g}f=-\frac{1}{\sqrt{g}}{\displaystyle \sum_{j,k=1}^{n}}\frac{\partial}{\partial x_{j}}\left(\sqrt{g}g^{jk}\frac{\partial(f\circ\phi^{-1})}{\partial x_{k}}\right)$.%
} est la généralisation du laplacien de $\mathbb{R}^{n}$. Pour une
fonction $f$ de classe $\mathcal{C}^{2}$ à valeurs réelles définie
sur une variété riemannienne $(M,g)$, et pour $\phi\,:\, U\subset M\rightarrow\mathbb{R}$
une carte locale de la variété $M$, l'opérateur de Laplace-Beltrami,
ou plus simplement laplacien de $(M,g)$, appliqué à la fonction $f$
est donné par la formule locale :\[
\Delta_{g}f=\frac{1}{\sqrt{g}}{\displaystyle \sum_{j,k=1}^{n}}\frac{\partial}{\partial x_{j}}\left(\sqrt{g}g^{jk}\frac{\partial(f\circ\phi^{-1})}{\partial x_{k}}\right)\]
où $g=\det(g_{ij})$ et $g^{jk}=(g_{jk})^{-1}$. Cet opérateur joue
un très grand rôle au sein même des mathématiques: son spectre est
un invariant géométrique majeur.

\subsection{Le contexte}

Considérons une variété riemannienne $(M,g)$ complète connexe de
dimension $n\geq1.$ On lui associe l'espace de Hilbert $L^{2}(M)=L^{2}(M,d\mathcal{V}_{g})$,
$\mathcal{V}_{g}$ désignant le volume riemannien associé à la métrique
$g$. L'opérateur de Schrödinger $H$ associé à la variété $(M,g)$
de potentiel $V$, $V$ étant une fonction de $M$ dans $\mathbb{R}$,
est défini comme l'opérateur linéaire non-borné sur les fonctions
lisses à support compacte $\mathcal{C}_{c}^{\infty}(M,\mathbb{R})$
par :\textit{\begin{equation}
H=-\frac{h^{2}}{2}\Delta_{g}+V\label{eq:}\end{equation}
}$\Delta_{g}$ étant le laplacien de $(M,g)$.

\subsection{Motivation}

On s'intéresse au problème spectral : Trouver les couples non-triviaux
$(\lambda,u)$ de scalaires complexes et de fonctions tels que :\[
-\Delta_{g}u+Vu=\lambda u\]

\begin{center}
(avec $u\in L^{2}(M)$ dans le cas non compact).
\par\end{center}

Dans le cas des variétés à bord on a besoin en supplément d'imposer
des conditions au bord sur les fonctions $u$, comme par exemple les
conditions de Dirichlet : on impose $u=0$ sur le bord de $M$, ou
celles de Neumann : $\frac{\partial u}{\partial n}=0$ sur le bord
de $M$, $n$ étant la normale extérieure au bord de $M$. Dans le
cas des variétés compactes sans bord, comme par exemple la sphère,
on parle de problème fermé. Il y a deux problématiques majeures liées
au spectre du laplacien (ou de l'opérateur de Schrödinger) sur une
variété riemannienne complète $(M,g)$:
\begin{enumerate}
\item les problèmes directs : étant donnée une variété riemannienne $(M,g)$,
que dire du spectre de l'opérateur $-\Delta_{g}$ ou de celui de l'opérateur
$-\Delta_{g}+V$ ?
\item Les problèmes inverses : étant donné le spectre de l'opérateur $-\Delta_{g}$,
que dire géométriquement de la variété $(M,g)$ ?
\end{enumerate}
Avant de répondre à ces questions, examinons quelques propriétés générales
du spectre.

\subsection{Le caractère auto-adjoint}

Une des premières questions à traiter lors de l'étude spectrale d'un
opérateur linéaire est celle du caractère auto-adjoint, ou à défaut
du caractère essentiellement auto-adjoint. Rappelons que un opérateur
linéaire $H$ est essentiellement auto-adjoint si son unique fermeture
${\displaystyle \overline{H}}$ est auto-adjointe. Quel est l'intérêt
du caractère auto-adjoint? Il y a au moins deux bonnes raisons d'en
parler : 
\begin{enumerate}
\item si $H$ est auto-adjoint, on a déjà une première information spectrale
importante : le spectre de l'opérateur $H$ est une partie de $\mathbb{R}$.
\item Le caractère auto-adjoint assure en mécanique quantique l'unicité
de la solution de l'équation de Schrödinger: en effet, à partir de
l'hamiltonien auto-adjoint $H$, on peut, via le calcul fonctionnel
construire de manière unique le groupe unitaire fortement continu
$\left\{ U(t)\right\} _{t\in\mathbb{R}}$ où : $U(t)=e^{-i\frac{t}{h}H}.$
\end{enumerate}
Quels sont les principaux résultats connus sur le caractère auto-adjoint? 
\begin{itemize}
\item Dans le cas où la variété $M=\mathbb{R}^{n}$ avec sa métrique standard,
T. Carleman \textbf{{[}28{]}} en 1934 à montré que si la fonction
$V$ est localement bornée et globalement minorée, alors l'opérateur
de Schrödinger $H$ est essentiellement auto-adjoint.
\item En 1972, T. Kato \textbf{{[}84{]}} a montré que l'on pouvait remplacer
dans l'énoncé de Carleman l'hypothèse $V\in L_{loc}^{\infty}(M)$
par $V\in L_{loc}^{2}(M)$.
\item En 1994, I. Olenik \textbf{{[}102{]}}, \textbf{{[}103{]}}, \textbf{{[}104{]}}
donne un énoncé très général concernant des variétés riemanniennes
complètes connexes quelconques avec des hypothèses plus complexes
sur la fonction $V.$ Un corollaire sympathique de cet énoncé est
le suivant :\end{itemize}
\begin{thm}
Soit $(M,g)$ une variété riemannienne complète connexe de dimension
$n\geq1$, et $V$ une fonction de $L_{loc}^{\infty}(M)$ tels que
$\forall x\in M,\, V(x)\geq C$, où $C$ est une constante réelle,
alors l'opérateur \[
H=-\Delta_{g}+V\]
est essentiellement auto-adjoint.
\end{thm}

\subsection{Le spectre de l'opérateur est-il discret ? }

Hormis le fait que le spectre est réel, que savons nous de plus ?
En 1934 K. Friedrichs \textbf{{[}64{]}} a montré que dans le cas où
la variété $M=\mathbb{R}^{n}$ avec sa métrique standard, si la fonction
$V$ est confinante, ie ${\displaystyle \lim_{|x|\rightarrow\infty}V(x)=+\infty}$,
alors le spectre de l'opérateur de Schrödinger $H$ est constitué
d'une suite de valeurs propres de multiplicités finies s'accumulant
en $+\infty$: \[
\lambda_{1}\leq\lambda_{2}\leq\cdots\leq\lambda_{k}\leq\cdots\]
Dans le contexte d'une variété riemannienne compacte avec un laplacien
pur ($V\equiv0$) nous savons aussi que le spectre de l'opérateur
$-\Delta_{g}$ est constitué d'une suite de valeurs propres positives,
de multiplicités finies, et s'accumulant en $+\infty$\[
0\leq\lambda_{1}\leq\lambda_{2}\leq\cdots\leq\lambda_{k}\leq\cdots\]
Qu'en est-il des variétés non compactes ? Commençons par donner une
définition :
\begin{defn}
Soit $(M,g)$ une variété lisse et $V$ une fonction de $M$ dans
$\mathbb{R}$, on dira que ${\displaystyle \lim_{|x|\rightarrow\infty}V(x)=+\infty}$,
si et seulement si\[
\forall A>0,\,\exists K\subset\subset M,\,\forall x\in M\smallsetminus K,\,\left|f(x)\right|\geq A.\]

\end{defn}
Un des théorèmes concernant le spectre de l'opérateur de Schrödinger
est celui de Kondratev et Shubin \textbf{{[}86{]}}, \textbf{{[}8}7\textbf{{]}}
qui donnent un énoncé assez technique sur les variétés à géométrie
bornée; de cet énoncé, on a le corollaire bien pratique suivant :
\begin{thm}
Soit $(M,g)$ une variété riemannienne complète connexe de dimension
$n\geq1$, et $V$ une fonction de $L_{loc}^{\infty}(M)$ telle que
${\displaystyle \lim_{|x|\rightarrow\infty}V(x)=+\infty}$. Alors
le spectre de l'opérateur $H=-\Delta_{g}+V$ est constitué d'une suite
de valeurs propres de multiplicités finies s'accumulant en $+\infty$\[
\inf_{x\in M}V(x)\leq\lambda_{1}\leq\lambda_{2}\leq\cdots\leq\lambda_{k}\leq\cdots\]

\end{thm}
Le théorème de Courant de 1953 \textbf{{[}50{]}}, assure en particulier
que la première valeur propre $\lambda_{1}$ de l'opérateur $H$ est
simple :\textit{\[
\inf_{x\in M}V(x)\leq\lambda_{1}<\lambda_{2}\leq\cdots\leq\lambda_{k}\leq\cdots\]
}

\subsection{Un aperçu sur les problèmes directs }

L'objectif est, à géométrie fixée, de pouvoir calculer, ou à défaut
de donner des propriétés sur le spectre de l'opérateur $-\Delta_{g}$
ou de celui de l'opérateur $-\Delta_{g}+V$. On va d'abord parler
de résultats exacts, puis de méthodes qualitatives.

\subsubsection*{Calcul explicite de spectre}

Il n'y a bien sur pas de méthodes générales pour calculer un spectre
d'opérateur linéaire; même dans le cas de Schrödinger sur une variété
raisonnable, le calcul est souvent difficile, et finalement on dispose
de peu d'exemples ou l'on peut expliciter complètement le spectre.
Voici tout de même quelques un exemple de calcul exact : l'oscillateur
harmonique, ou opérateur d'Hermite comme on le nomme en analyse harmonique.
C'est l'un des rares exemples d'opérateur de Schrödinger sur une variété
non compacte pour lequel on arrive à calculer explicitement son spectre.
L'oscillateur harmonique joue un rôle très important dans l'étude
des systèmes intégrables en classification symplectique : il sert
en effet de modèle de référence des équilibres stables de type elliptique;
pour plus de détails, on peut consulter le livre de Vu Ngoc \textbf{{[}San4{]}}.
Ici on prend $M=\mathbb{R}$ et :\textit{\[
H=-\frac{1}{2}\frac{d^{2}}{dx^{2}}+\frac{x^{2}}{2}.\]
}Les propriétés spectrales de l'opérateur $H$ sont très remarquables
: on arrive à calculer son spectre et les vecteurs propres associés
de manière explicite. Ces calculs, d'un point de vu très formel, se
trouvent dans n'importe quel bon livre de mécanique quantique. Pour
des démonstrations précises, on conseille par exemple le livre de
M.E. Taylor \textbf{{[}115{]}}. Le résultat est alors le suivant,
le spectre de l'opérateur $H$ est\textit{:\[
\sigma(H)=\left\{ n+\frac{1}{2},\, n\in\mathbb{N}\right\} \]
}avec comme vecteurs propres associés la base hilbertienne de $L^{2}(\mathbb{R})$
constituée des fonctions d'Hermite :\textit{\begin{eqnarray*}
e_{n}(x)=(2^{n}n!\sqrt{\pi})^{-\frac{1}{2}}e^{-\frac{x^{2}}{2}}H_{n}(x)\;\,\textrm{où}\,\; H_{n}(x)=(-1)^{n}e^{x^{2}}\frac{d^{n}}{dx^{n}}(e^{-x^{2}}) & .\end{eqnarray*}
}

\subsubsection*{Etude qualitative spectrale en bas du spectre}

Dans nombre de cas on ne sait pas calculer un spectre, on essaye alors
de le décrire de manière qualitative. Il y a disons deux sous thèmes
:
\begin{itemize}
\item le premier concerne le bas du spectre: on s'intéresse aux plus petites
valeurs propres de l'opérateur.
\item Le second est l'étude de l'asymptotique des grandes valeurs propres:
analyse semi-classique.
\end{itemize}
Donnons quelques exemples de résultats concernant le bas du spectre.
Commençons par des résultats de comparaison des premières valeurs
propres.
\begin{thm}
\textbf{(Théorème de Faber-Krahn, 1953).} Soit $M$ une partie bornée
de $\mathbb{R}^{n}$. En notant par $\lambda_{1}(M)$ la première
valeur propre de l'opérateur $-\Delta$ avec conditions de Dirichlet,
on a :\[
\lambda_{1}(M)\geq\lambda_{1}(B_{M})\]
$B_{M}$ désignant la boule euclidienne de volume égal à $\textrm{Vol}(M)$.
Et on a égalité si et seulement si $M$est isométrique à $B_{M}$.
\end{thm}
Dans le même style, on a aussi la version avec conditions de Neumann
où l'inégalité est dans l'autre sens:
\begin{thm}
\textbf{(Théorème de Szegö-Weinberger, 1954).} Soit $M$ une partie
bornée de $\mathbb{R}^{n}$. En notant par $\mu_{1}(M)$ la première
valeur propre de de l'opérateur $-\Delta$ avec conditions de Neumann,
on a :\[
\mu_{1}(M)\leq\mu_{1}(B_{M})\]
$B_{M}$ désignant la boule euclidienne de volume égal à $\textrm{Vol}(M)$.
Et on a égalité si et seulement si $M$est isométrique à $B_{M}$.
\end{thm}
Un autre type de résultat classique concerne les constantes de Cheeger
: soit $(M,g)$ une variété riemannienne connexe et compacte de dimension
$n\geq1.$ Pour toute partie bornée régulière $D$ de $M$, on considère
la quantité\[
h(D,g)=\frac{\textrm{Vol}(\partial D,g)}{\textrm{Vol}(D,g)}\]
où $\textrm{Vol}(\partial D,g)$ est le volume $n-1$ dimensionnel.
On définit ensuite la constante de Cheeger par \[
h(M,g)=\inf_{D\in X}h(D,g)\]
$X$ étant l'ensemble de tous les domaines de $M$ de volumes majorés
par $\frac{\textrm{Vol}(M,g)}{2}$. Alors un des résultats de Cheeger
est que la première valeur propre non nulle du Laplacien est minorée
par $\frac{h(M,g)^{2}}{4}$ \textbf{{[}19{]}}.\textbf{ }

Pour finir, donnons un autre résultat intéressant qui concerne la
multiplicité des valeurs propres en fonction de la topologie. Pour
cela plaçons nous un instant dans le cas des surfaces : si $(M,g)$
est une surface complète connexe, et \textit{\[
H=-\Delta_{g}+V\]
}avec $V\in\mathcal{C}^{\infty}(M,\mathbb{R})$ tels que ${\displaystyle \lim_{|x|\rightarrow\infty}V(x)=+\infty}$.
En notant (cf. Théorème 2) par \[
\lambda_{1}<\lambda_{2}\leq\cdots\leq\lambda_{k}\leq\cdots\]
le spectre de l'opérateur $H$ et par $m_{k}$ la multiplicité de
la $k$-ème valeur propre $\lambda_{k}$, nous avons le résultat dû
à S. Y. Cheng \textbf{{[}32{]}} et amélioré par G. Besson \textbf{{[}21{]}}
, Y. Colin De Verdière \textbf{{[}38{]}}, N. Nadirashvili \textbf{{[}101{]}}
et B. Sévennec \textbf{{[}113{]}} :
\begin{thm}
Sous les hypothèses précédentes nous avons :

- Si $X=\mathbb{S}^{2}$ ou $\mathbb{R}^{2}$, alors pour tout $k\geq3,\, m_{k}\leq2k-3$.

- Si $X=\mathbb{P}^{2}(\mathbb{R})$ ou $K_{2}$  (la bouteille de
Klein), alors pour tout $k\geq1,\, m_{k}\leq2k+1$.

- Si $X=\mathrm{T}^{2}$, alors pour tout $k\geq1,\, m_{k}\leq2k+2$.

- En notant par $\chi(M)$ la caractéristique d'Euler-Poincaré, si
$\chi(M)<0$, alors pour tout $k\geq1,\, m_{k}\leq2k-2\chi(M)$.
\end{thm}

\subsubsection*{Etude qualitative spectrale en haut de spectre}

L'exemple de base est la formule asymptotique de Weyl de 1911,\textbf{
}{[}\textbf{19{]}}. Pour le laplacien dans un domaine rectangulaire
$\Omega$ de $\mathbb{R}^{2}$ avec des conditions de Dirichlet aux
bords, le physicien P. Debye conjectura que le nombre de valeurs propres
$\mathcal{N}(\lambda)$ inférieure à un réel positif $\lambda$, vérifie
l'équivalence, pour $\lambda\rightarrow+\infty$\textit{\[
{\displaystyle \mathcal{N}(\lambda)\sim}\frac{\textrm{Vol}(\Omega)}{4\pi}\lambda\]
}où$\textrm{Vol}(\Omega)$ est l'aire du rectangle $\Omega$. En 1911,
H. Weyl démontra cette conjecture.
\begin{thm}
Soit $(M,g)$ une variété riemannienne compacte connexe de dimension
$n$, si on note par $\lambda_{1}<\lambda_{2}\leq\cdots\leq\lambda_{k}\leq\cdots$
les valeurs propres de l'opérateur $-\Delta_{g}$ sur $M$, on a l'équivalent
pour $\lambda\rightarrow+\infty$\[
\mathrm{Card}\left(\left\{ k\in\mathbb{N},\,\lambda_{k}\leq\lambda\right\} \right){\displaystyle \sim}\frac{B_{n}\textrm{Vol}(M,g)}{(2\pi)^{n}}\lambda^{\frac{n}{2}}\]
où $B_{n}=\frac{\pi^{\frac{n}{2}}}{\Gamma\left(\frac{n}{2}+1\right)}$
est le volume de la boule unité de $\mathbb{R}^{n}$.
\end{thm}
On reviendra dans la dernière partie à l'étude du ''haut'' du spectre
en utilisant l'analyse semi-classique.

\subsection{Problèmes inverses : la géométrie spectrale}

\subsubsection*{Le son détermine t-il la forme d'un tambour? }

La problématique inverse est la suivante : étant donné le spectre
d'un laplacien ou d'un opérateur de Schrödinger, quelles informations
géométriques sur la variété $(M,g)$ peut-on avoir? Dans le cas du
laplacien, un des premiers à formaliser mathématiquement cette question
est sans doute M. Kac \textbf{{[}83{]}} en 1966 dans son célèbre article
''Can one hear the shape of a drum?''%
\footnote{''Peut-on entendre la forme d'un tambour?''%
}: pour le laplacien riemannien, une suite de valeurs propres (un ensemble
d'harmoniques du tambour) caractérise-t'elle, à isométrie près, la
variété de départ (la géométrie du tambour) ? Il est connu que si
deux variétés sont isométriques, elles sont alors isospectrales (c'est-à-dire
ont le même spectre). Mais qu'en est t-il de la réciproque ? 

On sait depuis 1964, que la réponse au problème de M. Kac est négative;
en effet, J. Milnor \textbf{{[}98{]}} donne comme exemple de variétés
isospectrales mais non isométriques, une paire de tores plats de dimension
16. Depuis, de nombreux autres exemples ont été trouvés, a commencer
par T. Sunada \textbf{{[}114{]}}, qui en 1985 donne une méthode de
construction systématique de variétés isospectrales non isomorphes.
C. Gordon et E.N. Wilson \textbf{{[}69{]}} ont aussi donné en 1984
une méthode de construction de déformations continues de variétés
qui sont isospectrales sans être isométriques. L'histoire ne s'arrête
pas là, d'autres méthodes de construction apparaissent, comme par
exemple la méthode de transplantation de P. Bérard \textbf{{[}12{]}},
\textbf{{[}13{]}} etc...

En 1992, C. Gordon, D. Webb et S. Wolpert \textbf{{[}70{]}} donnent
le premier exemple de deux domaines plans non isométriques, mais ayant
tout de même un spectre commun pour le laplacien avec conditions de
Neumann ou de Dirichlet.

\begin{center}
\includegraphics[scale=0.33]{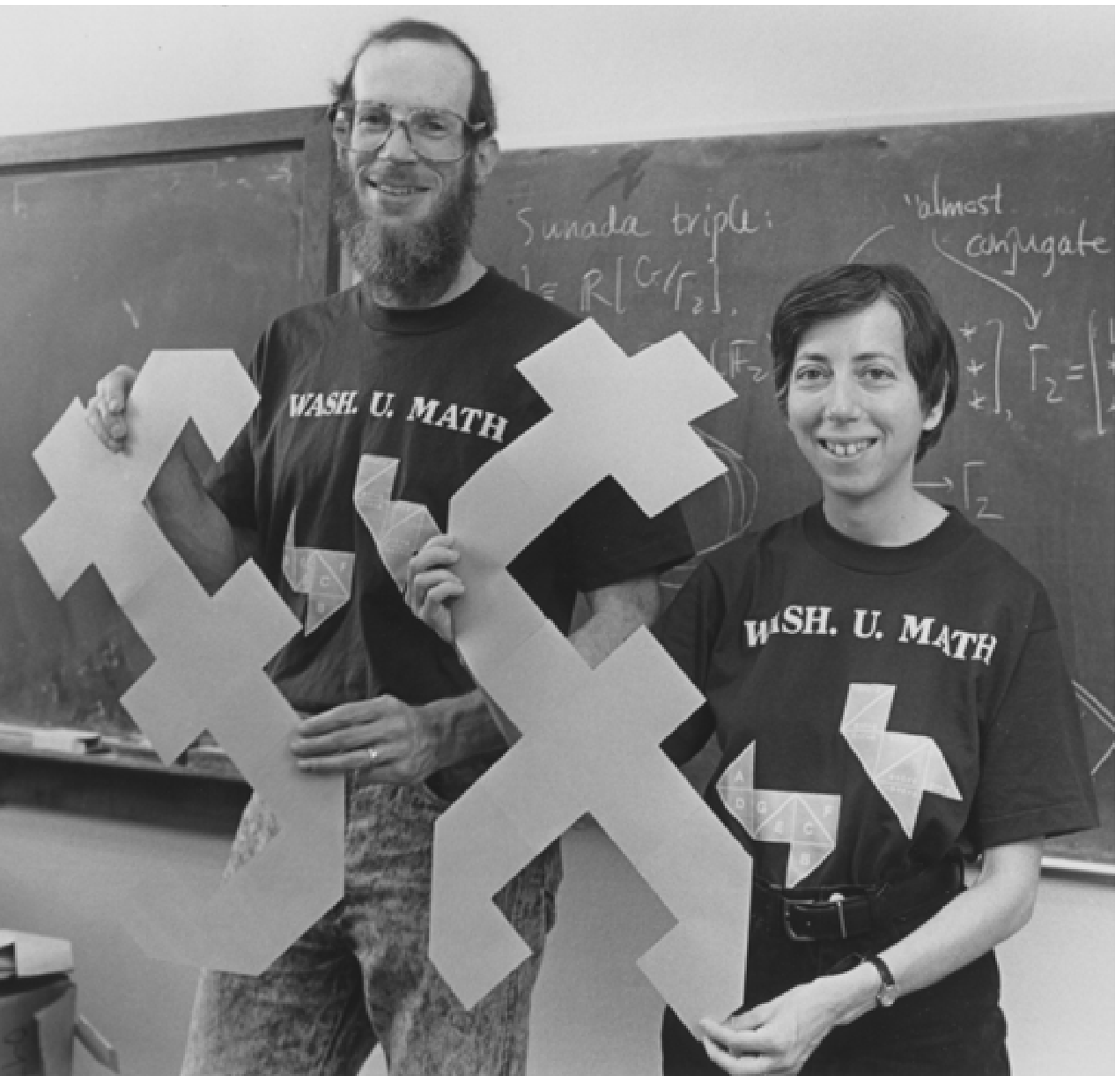}
\par\end{center}

\begin{center}
\textit{Fig 1. Une photographie de C. Gordon et D. Webb avec leurs
fameux domaines plans isospectraux mais non isométriques.}
\par\end{center}

\vspace{+0.25cm}

Pour plus de détails sur cet exemple ou pourra consulter les auteurs
\textbf{{[}70{]}}, \textbf{{[}71{]}} mais aussi voir les articles
très pédagogiques de P. Bérard \textbf{{[}15{]}}, \textbf{{[}16{]}},
\textbf{{[}17{]}}. Mentionnons aussi le travail de S. Zelditch \textbf{{[}121{]}}
datant de 2000, où il montre que si on se restreint à des parties
de $\mathbb{R}^{2}$ simplements connexes avec un bord analytique
et possédant deux axes de symétrie orthogonaux, alors le spectre détermine
complètement la géométrie.

\subsubsection*{Spectre des longueurs et formules de traces}

La donnée du spectre du laplacien donne des informations sur d'autres
invariants géométriques comme la dimension, le volume et l'intégrale
de la courbure scalaire. En fait le laplacien fournit aussi d'autres
invariants, comme par exemple le spectre des longueurs d'une variété.
Le spectre des longueurs d'une variété riemannienne est l'ensemble
des longueurs des géodésiques périodiques. En 1973 Y. Colin de Verdière
\textbf{{[}33{]}}, \textbf{{[}34{]}} montre que dans le cas compact,
modulo une hypothèse de généricité toujours vérifiée à courbure sectionnelle
négative, le spectre du laplacien détermine complètement le spectre
des longueurs. La technique utilisé par Y. Colin de Verdière repose
sur les formules de traces. Ces dernières s'utilisent dans un cadre
beaucoup plus général que celui des opérateurs de Schrödinger.

Le principe formel des formules de traces est le suivant: considérons
d'abord un opérateur linéaire $H$ non-borné sur un Hilbert ayant
un spectre discret : $\sigma(H)=\left\{ \lambda_{n},\, n\geq1\right\} $,
et puis une fonction $f$ ''sympathique''. La formule de trace consiste
alors à calculer la trace de l'opérateur $f(H)$ de deux façons différentes
:
\begin{itemize}
\item la première façon, lorsque que cela a un sens, avec les valeurs propres
de l'opérateur linéaire $f(H)$ :\textit{\[
\textrm{Tr}(f(H))={\displaystyle \sum_{k\geq1}f(\lambda_{k})}.\]
}
\item La seconde façon, avecle noyau de Schwartz de l'opérateur $f(H)$\textit{:}
si \textit{$f(H)\varphi(x)={\displaystyle {\displaystyle \int_{M}K_{f}(x,y)\varphi(y)\, dy}}$,}
alors\textit{\[
\textrm{Tr}(f(H))={\displaystyle {\displaystyle \int_{M}K_{f}(x,x)\, dx}}.\]
}
\end{itemize}
Ainsi\textit{\[
{\displaystyle \sum_{k\geq1}f(\lambda_{k})}={\displaystyle {\displaystyle \int_{M}K_{f}(x,x)\, dx}.}\]
}La difficulté réside dans le choix de $f$, d'une part pour légitimer
ces formules, et d'autre part pour arriver à en tirer des informations
spectro-géométriques. Les choix de fonctions $f$ les plus courants
sont : $f(x)=e^{-xt}$ où $t\geq0$ (fonction de la chaleur), $f(x)=\frac{1}{x^{s}}$,
où $s\in\mathbb{C}$, avec $\textrm{Re}(s)>1$ (fonction zêta de Riemann),
$f(x)=e^{-\frac{itx}{h}}$ où $t\geq0$ (fonction de Schrödinger),
etc ...

Pour fixer les idées, donnons un exemple simple de formule de trace
exacte : la formule sommatoire de Poisson pour un réseau $\Gamma$
de $\mathbb{R}^{n}$. La formule de Poisson sur le tore $\Gamma\setminus\mathbb{R}^{n}$
nous donne l'égalité :

\textit{\begin{equation}
{\displaystyle \sum_{\lambda\in\sigma(\Delta_{g})}e^{-\lambda t}}=\frac{\textrm{Vol}\left(\Gamma\setminus\mathbb{R}^{n}\right)}{(4\pi t)^{\frac{n}{2}}}{\displaystyle \sum_{l\in\Sigma}e^{-\frac{l^{2}}{4t}}}\label{eq:}\end{equation}
}où $\sigma(\Delta_{g})$ est le spectre de l'opérateur $\Delta_{g}$
et $\Sigma$ le spectre des longueurs comptés aves leurs multiplicités,
la multiplicité d'une longueur étant le nombre de classes d'homotopies
de lacets du tore plat $\Gamma\setminus\mathbb{R}^{n}$ représentées
par une géodésique périodique de cette longueur. Dans l'égalité (4.2)
le terme de droite correspond à la partie géométrique (volume, dimension,...)
alors que le terme de gauche contient les informations spectrales.
Pour une référence récente voir \textbf{{[}47{]}}.

\subsection{Métrique de Agmon et puits multiples}

\subsubsection*{Définition de la métrique}

Dans tout la suite $(M,g)$ est, soit une variété riemannienne compacte,
ou bien l'espace $\mathbb{R}^{n}$ tout entier. L'opérateur de Schrödinger
$P_{h}$ associé à la variété $(M,g)$ de potentiel $V$, $V$ étant
une fonction de $M$ dans $\mathbb{R}$, est défini comme l'opérateur
linéaire non-borné sur les fonctions lisses à support compact $\mathcal{C}_{c}^{\infty}(M,\mathbb{R})$
par \textit{\begin{equation}
P_{h}=-\frac{h^{2}}{2}\Delta_{g}+V\label{eq:}\end{equation}
}avec une fonction $V$ est localement bornée, globalement minorée
et confinante :${\displaystyle \lim_{|x|\rightarrow\infty}V(x)=+\infty}$;
ce qui assure un spectre réel discret constitué de valeurs propres.
Soit $E>0$; dans le cas où $M=\mathbb{R}^{n}$ on suppose en outre
sur la fonction $V$ que ${\displaystyle \lim_{|x|\rightarrow0}V(x)-E>0.}$
\begin{defn}
La métrique de Agmon est donnée par $\max\left(V(x)-E,0\right)dx^{2}$
où $dx^{2}=g(x)$; $g$ est la métrique de la variété $(M,g)$. 
\end{defn}
Ainsi l'application $m_{A}:\, x\in M\mapsto\max\left(V(x)-E,0\right)g(x)$
est un produit scalaire sur $T_{x}M`\times T_{x}M.$
\begin{rem}
C'est une métrique dégénérée sur $(M,g)$.\end{rem}
\begin{defn}
La distance de Agmon sur $(M,g)$ est définie par : $d_{A}(a,b)={\displaystyle \inf_{\gamma:a\rightarrow b}\mathcal{L}(\gamma)}$
où\[
\mathcal{L}(\gamma)={\displaystyle \int_{0}^{1}\sqrt{m_{A}(\gamma(t)\left(\dot{\gamma}(t);\dot{\gamma}(t)\right)}\, dt};\]
$\gamma:\left[0,1\right]\rightarrow M$ est un arc de classe $\mathcal{C}^{1}$
de $M$ tel que $\gamma(0)=a$ et $\gamma(1)=b$. 
\end{defn}
Si on se place sur $M=\mathbb{R}^{n}$ et si on considère deux réels
$a$ et $b$ avec $a<b$ nous avons que $d_{A}(a,b)={\displaystyle \int_{0}^{1}\sqrt{m_{A}(\gamma(t)\left(\dot{\gamma}(t);\dot{\gamma}(t)\right)}\, dt}$
avec $\gamma(t)=at+(1-t)b,\; t\in$$\left[0,1\right];$ ainsi $\dot{\gamma}(t)=a-b;$
et donc \[
d_{A}(a,b)={\displaystyle \int_{0}^{1}\sqrt{V(\gamma(t))-E\left|a-b\right|}\, dt}\]
\[
={\displaystyle \int_{a}^{b}\sqrt{V(x)-E}\, dx.}\]

\subsubsection*{Application aux puits multiples}

Pour plus de commodité on prendra $E=0$ et on suppose que la fonction
potentiel à $N$ puits :\[
\left\{ x\in M;\, V(x)\leq0\right\} ={\displaystyle \coprod_{j=1}^{N}U_{j}}\]
où $N$ est un entier et les $U_{j}$ sont des compacts de $M$ (on
parle de puits). Avec la distance $d_{A}$ on peut définir $S_{0}$:
la distance de Agmon entre deux puits\[
S_{0}:={\displaystyle \min_{j\neq k}d_{A}\left(U_{j},U_{k}\right)}.\]
Soit $s\in]0,S_{0}[$ et posons \[
B_{A}\left(U_{j},s\right):=\left\{ x\in M,\, d_{A}\left(x,U_{j}\right)<s\right\} ;\]
alors si $j\neq k$ on a $\overline{B\left(U_{j},s\right)}\cap U_{k}=\emptyset$
et on peut trouver des variétés compactes à bords $\mathcal{C}^{2}$-lisses
$\left(M_{j}\right)_{j}\subset M$ telles que $\overline{B\left(U_{j},s\right)}\subset\textrm{Int}\left(M_{j}\right)$
et si $j\neq k$ alors $M_{j}\cap U_{k}=\emptyset$. Sur les hilberts
$L^{2}\left(M_{j}\right)$ on considère les restrictions (avec conditions
de Dirichlet) auto-adjointes de $P_{h}$ à $M_{j}$, notée $P_{j}$.
Soit $I(h)=\left[\alpha(h),\beta(h)\right]$ un intervalle qui tend
vers le singleton $\{0\}$ quand $h\rightarrow0$. On a le :
\begin{thm}
\textbf{{[}77{]}} Il existe $s_{1}<s_{0}$ tel que pour $h$assez
petit, il existe une bijection $b\,:\,\sigma\left(P_{h}\right)\cap I(h)\rightarrow{\displaystyle \coprod_{j=1}^{N}\left(\sigma\left(P_{j}\right)\cap I(h)\right)}$
telle que pour tout $\tau<s_{1}$ on ait :\[
b(\lambda)-\lambda=O\left(e^{-\frac{\tau}{h}}\right).\]

\end{thm}
Ce théorème explique le phénomène de séparation des valeurs propres
dans chaques puits; avec une distance exponentiellement petite.

\section{Quantification et limite semi-classique}

\subsection{Problématique}

Le principe de correspondance est à la base des postulats de la mécanique
quantique, c'est le ''dictionnaire'' entre le monde classique et
le monde quantique. La quantification est la théorie mathématique
qui a pour but d'essayer de justifier ce dictionnaire; plus précisément
d'essayer de construire un morphisme entre ces deux mondes. La quantification
est le passage du classique au quantique; l'opération inverse est
qualifiée de limite semi-classique. Pour plus de détails sur la quantification
voir par exemples les livres \textbf{{[}60{]}} et \textbf{{[}30{]}}.

\vspace{0.25cm}

\begin{center}
\begin{tabular}{|c|c|}
\hline 
\textbf{Mécanique classique} & \textbf{Mécanique quantique}\tabularnewline
\hline
\hline 
$\left(M,\omega\right)$ variété symplectique & $\mathcal{H\subset}L^{2}(X)$, où $X$ variété .\tabularnewline
\hline 
Points $x\in M$.  & Vecteurs $\varphi\in\mathcal{H}$.\tabularnewline
\hline 
Algèbre $\mathcal{C}^{\infty}(M).$  & Algèbre d'opérateurs sur $\mathcal{H}$.\tabularnewline
\hline 
Crochet de Poisson$\left\{ .\right\} $. & Commutateur $\left[.\right]$.\tabularnewline
\hline
\hline 
Équation de Hamilton. & Équation de Schrödinger. \tabularnewline
\hline
\end{tabular}
\par\end{center}

\vspace{0.25cm}

On voit très nettement les premières grosses difficultés mathématiques
du passage d'un monde à l'autre : passage de la dimension finie à
infinie, passage du commutatif au non commutatif, la linéarité qui
apparaît en mécanique quantique, etc...

\subsection{Impossibilité de la quantification idéale}

Commerçons par le passage de la mécanique classique à la mécanique
quantique; la question mathématique précise de ce passage se formule
par l'existence de :
\begin{defn}
On appelle quantification (idéale) de la variété symplectique $(M,\omega)$
toute application linéaire :\[
\mathcal{\mathbf{Q}}\,:\,\mathcal{C}^{\infty}(M)\rightarrow\textrm{\{Algèbre d'opérateurs sur un Hilbert\}}\]
vérifiant les quatres axiomes suivants:\[
\mathbf{(i)}\;\;\mathcal{\mathbf{Q}}(1)=\mathbb{I}_{d}\]
\[
\mathbf{(ii)}\;\;\mathcal{\mathbf{Q}}\left(\left\{ f,g\right\} \right)=\frac{i}{h}\left[\mathcal{\mathbf{Q}}(f),\mathcal{\mathbf{Q}}(g)\right]\]
\[
\mathbf{(iii)}\;\;\mathcal{\mathbf{Q}}(x_{k})=x_{k}\textrm{ et }\mathbf{\mathcal{\mathbf{Q}}}(\xi_{k})=-ih\frac{\partial}{\partial x_{k}}\]
\[
\mathbf{(iv)}\;\;\left(\mathbf{\mathcal{\mathbf{Q}}}(f)\right)^{*}=\mathcal{\mathbf{Q}}(\overline{f}).\]

\end{defn}
En fait si on se place sur la variété symplectique la plus simple
possible : $\mathbb{R}^{2n}$ il n'existe pas de quantification idéale:
en effet on a le fameux (voir par exemple \textbf{{[}61{]}}):
\begin{thm}
\textbf{(Van Hove, 1952).} Il n'existe pas de quantification :\[
\mathcal{\mathbf{Q}}\,:\,\mathbb{R}[x_{1},x_{2},\ldots,x_{n},\xi_{1},\xi_{2},\ldots,\xi_{n}]\rightarrow\textrm{\{Algèbre d'opérateurs sur $L^{2}(\mathbb{R}^{n})$\}.}\]
 
\end{thm}
Il faut alors affaiblir la définition de la quantification idéale,
en particulier l'axiome \textbf{(ii)}. Cela est possible dans le cas
de $\mathbb{R}^{2n}$. On verra ça dans la partie 7 à l'aide des opérateurs
pseudo-différentiels.

\subsection{Principe de l'analyse semi-classique}

De manière extrêmement simple et naïve l'idée de l'analyse semi-classique
est de comprendre le quantique lorsque le paramètre $h\rightarrow0$.
Pour le lecteur qui voudrait en savoir plus sur l'analyse semi-classique,
on conseille la littérature suivante : Y. Colin de Verdière \textbf{{[}46{]}},
Dimassi-Sjöstrand \textbf{{[}52{]}}, L. Evans et M. Zworski \textbf{{[}58{]}},
A. Martinez \textbf{{[}94{]}}, D. Robert \textbf{{[}109{]}}, S. Vu
Ngoc \textbf{{[}119{]}}. Revenons un instant à la limite $h\rightarrow0$,
quel est son sens physique? En ''principe'' tout système physique
est de par nature quantique. D'après les fameuses inégalités d'incertitude
de Heisenberg, on ne peut pas mesurer précisément à la fois vitesse
et position d'un électron, sauf si $h=0$. En fait plus $h$ est petit,
plus on peut faire des mesures simultanées précises. Ainsi plus $h\rightarrow0$,
plus on se rapproche du déterminisme de la mécanique classique sur
le fibré cotangent $T^{*}M$. En pratique quand on fait de l'analyse
semi-classique, on travaille à la fois avec des objets classiques
(variétés symplectiques, algèbre des fonctions $\mathcal{C}^{\infty}$,
crochet de Poisson, équations de Hamilton,...) et des objets quantiques
(espace de Hilbert, algèbre d'opérateurs, commutateur, équation de
Schrödinger,...). 

Une autre philosophie de l'analyse semi-classique est la suivante
: dans la limite des grandes valeurs propres, le spectre de l'opérateur
de Schrödinger sur une variété riemanienne $(M,g)$\textit{\[
H=-\frac{h^{2}}{2}\Delta_{g}+V\]
}ou plus généralement d'un opérateur pseudo-différentiel, est remarquablement
liée à une géométrie sous-jacente. Celle-ci vit sur le fibré cotangent
$T^{*}M$, vu comme une variété symplectique: c'est la géométrie de
l'espace des phases. C'est d'ailleurs le même phénomène qui permet
de voir la mécanique classique (structure de variété symplectique)
comme limite de la mécanique quantique (structure d'algèbre d'opérateurs).

Voyons pourquoi s'intéresser à l'asymptotique du spectre de l'opérateur
\textit{$H$,} revient dans une certaine mesure à faire tendre le
paramètre $h$ vers $0$ (limite semi-classique). Par exemple, pour
$E>0$ fixé, l'équation :\[
-\frac{h^{2}}{2}\Delta_{g}\varphi=E\varphi\]
admet $\varphi_{k}$, le $k$-ième vecteur propre du laplacien $\Delta_{g}$,
comme solution si\[
-\frac{h^{2}}{2}\lambda_{k}=E.\]
Ainsi si $h\rightarrow0^{+}$, alors $\lambda_{k}\rightarrow+\infty$.
C'est pourquoi la limite semi-classique peut aussi se voir comme l'asymptotique
des grandes valeurs propres du laplacien.

\section{Opérateurs pseudo-différentiels}

\subsection{De Fourier à nos jours...}

Historiquement on peut dire que c'est Fourier en utilisant la transformée
de Fourier pour résoudre l'équation de la chaleur $\frac{\partial u}{\partial t}=\triangle u$
qui fut le pionnier des opérateurs pseudo-différentiels. Dans le début
des annés 60, Caldéron et Zygmund introduisent les opérateurs intégraux
a noyaux singuliers en vue de résoudre des équations aux dérivées
partielles avec des coefficients variables. Entre 60 et 70, Niremberg
et Hörmander ont donné la théorie des opérateurs pseudo-différentiels
sans le paramètre semi-classique $h$. Un peu plus tard Maslov, Helffer,
Robert et Sjöstrand se sont intéresses à l'étude avec $h$. Il y en
à en réalité moulte opérateurs pseudo-différentiels, destinés à tel
ou tel application et il est difficile de présenter une théorie très
générale. Dans la suite, on va rapidement décrire ''une'' théorie
des opérateurs pseudo-différentiels avec le paramètre semi-classique
$h$. De nos jours la théorie des opérateurs pseudo-différentiels
est devenu un outil d'analyse très puissant, en particulier en équations
aux dérivées partielles, en analyse sur les variétés et même en géométrie
algébrique complexe. On va rappeler brièvement une des définitions
des opérateurs pseudo-différentiels : celle de la quantification de
Weyl sur $\mathbb{R}^{2n}$ (ou sur $T^{*}X$) et donner les principales
propriétés. de ce type d'opérateurs pseudo-différentiels. Pour plus
de détails, voir \textbf{{[}52{]}}.

\subsection{Symboles et transformée de Weyl}

De manière très formelle, la quantification de Weyl consite à associer
à une fonction symbole convenable $a:\,(x,\xi)\mapsto a(x,\xi)\in\mathcal{C}^{\infty}(\mathbb{R}^{2n})$
un opérateur linéaire $\mathbf{O}_{p}^{w}(a)$ de $\mathcal{S}(\mathbb{R}^{n})$
dans lui même, admettant une représentation intégrale :pour toute
fonction $u\in\mathcal{S}(\mathbb{R}^{n})$ et pour tout $x\in\mathbb{R}^{n}$:

\[
\left(\mathbf{O}_{p}^{w}(a)(u)\right)(x):=\frac{1}{(2\pi h)^{n}}\int\int_{\mathbb{R}^{n}\times\mathbb{R}^{n}}e^{\frac{i}{h}(x-y)\xi}a\left(\frac{x+y}{2},\xi\right)u(y)\, dyd\xi.\]
Une des premières difficulté des la théorie des opérateurs pseudo-différentiels
est de donner un sens à ce type de formule. Bien sur pour des symboles
$a\in\mathcal{S}(\mathbb{R}^{2n})$ c'est plutôt facile à définir,
mais pour faire de la quantification il faut au moins autoriser des
symboles polynomiaux en $x$ et $\xi$. On va définir une classe de
symboles usuelle pour faire de la quantification : sur la variété
$X:=\mathbb{R}^{n}$, et pour $k,m\in\mathbb{Z}^{2}$, on définit
l'ensemble de symboles d'indice $k$ et de poids $\left\langle z\right\rangle ^{m}$
sur la variété $X$ où $\left\langle z\right\rangle =\left\langle x,\xi\right\rangle :=\left(1+|z|^{2}\right)^{\frac{1}{2}}$,
par :\[
S^{k}\left(X,\left\langle z\right\rangle ^{m}\right)\]
\[
:=\left\{ a_{h}(z)\in\mathcal{C}^{\infty}\left(T^{*}X\right),\,\forall\alpha\in\mathbb{N}^{n},\,\exists C_{\alpha}\geq0,\,\forall z\in T^{*}X,\,\left|{\displaystyle \partial_{z}^{\alpha}}a_{h}(z)\right|\leq C_{\alpha}h^{k}\left\langle z\right\rangle ^{m}\right\} .\]
Avec des intégrations par parties habiles (technique des intégrales
oscillantes), voir \textbf{{[}52{]}} ou \textbf{{[}94{]}} on montre
que :
\begin{thm}
Si $a\in S^{k}\left(X,\left\langle z\right\rangle ^{m}\right)$, avec
$m\geq0$, alors $\mathbf{O}_{p}^{w}(a)$ est un opérateur linéaire
de $\mathcal{S}(\mathbb{R}^{n})$ dans lui même.
\end{thm}
Donnons un exemple très important de calcul de transformée de weyl
: celles des variables canoniques.
\begin{example}
\textit{Le quantifié de Weyl de la fonction $(x,\xi)\mapsto1$ est
l'opérateur identité. Le quantifié de Weyl de la fonction $(x,\xi)\mapsto x_{j}$
est l'opérateur de multiplication par la variable $x_{j}$. Le quantifié
de Weyl de la fonction $(x,\xi)\mapsto\xi_{j}$ est l'opérateur de
dérivation $-ih\frac{\partial}{\partial x_{j}}.$}
\end{example}
En analyse semi-classique, on est aussi amené à considérer des symboles
ayant des développements asymptotiques en puissance de $h$: soit
$a_{h}\in S^{0}\left(X,\left\langle z\right\rangle ^{m}\right)$,
on dira que ce symbole est classique si et seulement s'il existe une
suite de symboles $\left(a_{j}\right)_{j\in\mathbb{N}}\in S^{0}\left(X,\left\langle z\right\rangle ^{m}\right)^{\mathbb{N}}$
indépendant de $h$ tels que pour tout $k^{\prime}\geq0$, on ait
:\[
\left(a_{h}(z)-{\displaystyle \sum_{j=0}^{k^{\prime}}a_{j}(z)h^{j}}\right)\in S^{k^{\prime}+1}\left(X,\left\langle z\right\rangle ^{m}\right).\]
On note alors $a_{h}={\displaystyle \sum_{j=0}^{+\infty}a_{j}h^{j}}$,
on dira aussi que $a_{0}$ est le symbole principal de $a_{h}$. Le
théorème de resommation de Borel (voir\textbf{ {[}Mar{]}}) assure
que pour toute suite arbitraire de symboles $\left(a_{j}\right)_{j\in\mathbb{N}}\in S\left(X,\left\langle z\right\rangle ^{m}\right)$,
il existe une unique, modulo $O(h^{\infty})$, fonction symbole $a_{h}$
telle que $a_{h}={\displaystyle \sum_{j=0}^{+\infty}a_{j}h^{j}}$
.

\subsection{Quelques propriétés}

Si on fait le produit de deux opérateurs pseudo-différentiels de symboles
$a$ et $b$, alors l'opérateur produit est encore un opérateur pseudo-différentiel
:
\begin{thm}
Quel que soient les symboles $(a,b)\in S\left(X,\left\langle \xi\right\rangle ^{m_{1}}\right)\times S\left(X,\left\langle \xi\right\rangle ^{m^{\prime}}\right)$,
il existe un symbole $c\in S\left(X,\left\langle \xi\right\rangle ^{m+m^{\prime}}\right)$
tels que \[
\mathbf{O}_{p}^{w}(a)\circ\mathbf{O}_{p}^{w}(b)=\mathbf{O}_{p}^{w}(c).\]
 De plus un choix possible pour le symbole $c$ est donné par la formule
de Moyal :\[
c=a\star b:=\sum_{j=0}^{\infty}\frac{h^{j}}{j!(2i)^{j}}a\left(\sum_{p=1}^{n}\overleftarrow{\partial_{\xi_{p}}}\overrightarrow{\partial_{x_{p}}}-\overleftarrow{\partial_{x_{p}}}\overrightarrow{\partial_{\xi_{p}}}\right)^{j}b\]
où la flèche indique sur quelle fonctions, $a$ ou $b$ la dérivée
doit opérer.
\end{thm}
Citons maintenant un théorème de continuité fondamental: 
\begin{thm}
(Calderon-Vaillancourt) Si le symbole $a\in S^{k}\left(X,1\right),\, k\geq1$,
alors l'opérateur $\mathbf{O}_{p}^{w}(a)$ est un opérateur linéaire
continu de \textup{$L^{2}(\mathbb{R}^{n})$ }dans\textup{ $L^{2}(\mathbb{R}^{n})$
}:\[
\exists C,M>0,\;|||\mathbf{O}_{p}^{w}(a)|||\leq C\sum_{|\alpha|\leq M}||\partial^{\alpha}a||_{L^{\infty}(\mathbb{R}^{2n})}.\]

\end{thm}
Une des principales applications de ce théorème de continuité est
l'inversion des opérateurs pseudo-différentiels : un symbole $a\in S\left(X,\left\langle z\right\rangle ^{m}\right)$
est elliptique en $(x_{0},\xi_{0})\in T^{*}X$ si et seulement si
$|a(x_{0},\xi_{0})|\neq0$.
\begin{thm}
Soit $m\in\mathbb{R}$ et $a\in S\left(X,\left\langle \xi\right\rangle ^{m}\right)$
elliptique sur $T^{*}X$, alors il existe un symbole $b\in S\left(X,\left\langle \xi\right\rangle ^{-m}\right)$
tels que : \[
\mathbf{O}_{p}^{w}(a)\circ\mathbf{O}_{p}^{w}(b)=I_{d}+R_{1}\; et\;\mathbf{O}_{p}^{w}(b)\circ\mathbf{O}_{p}^{w}(a)=I_{d}+R_{2}\]
où $R_{1},R_{2}$ sont des opérateurs linéaire continus de \textup{$L^{2}(\mathbb{R}^{n})$
}dans\textup{ $L^{2}(\mathbb{R}^{n})$} et vérifiant $|||R_{1}|||+|||R_{2}|||=O(h^{\infty})$.
\end{thm}

\section{Analyse microlocale}

Il semble que l'analyse microlocale prenne naissance dans les années
70 pour des applications en équations aux dérivés partielles. Aujourd'hui
l'analyse microlocale est présente dans beaucoup de mathématiques,
comme la topologie, la géométrie analytique,... 

Le principe moral de l'analyse microlocale est alors d'utiliser des
variétés symplectique pour faire de l'analyse : pour des objets vivants
naturellement sur une variété quelconque $X$, on peut aussi les faire
vivre sur la variété $T^{*}X$. 

Donnons ici quelques éléments d'analyse microlocale, pour plus de
détails voir par exemple\textbf{ {[}118{]}, {[}119{]}}, \textbf{{[}46{]}}
ou \textbf{{[}73{]}}.

\subsection{Fonctions admissibles}

Pour $h_{0}>0$ fixé, l'ensemble\[
A:=\left\{ \lambda(h)\in\mathbb{C}^{\left]0,h_{0}\right]},\,\exists N\in\mathbb{Z},\;\left|\lambda(h)\right|=O(h^{-N})\right\} \]
est un anneau commutatif pour les opérations usuelles sur les fonctions.
On voit aussi sans peine que \[
I:=\left\{ \lambda(h)\in A,\,\lambda(h)=O(h^{\infty})\right\} \]
 est un idéal bilatère de $A$, on définit alors l'anneau $\mathbb{C}_{h}$
des constantes admissibles comme étant l'anneau quotient $A/I$.

On peut alors définir le $\mathbb{C}_{h}$-module des fonctions admissibles
:
\begin{defn}
L'ensemble $\mathcal{A}_{h}(X)$ des fonctions admissibles sur $X$
est l'ensemble des distributions $u_{h}\in\mathcal{D}^{\prime}(X)$
tels que pour tout opérateur pseudo-différentiel $P_{h}$ dont le
symbole dans une carte locale est a support compact\[
\exists N\in\mathbb{Z},\;\left\Vert P_{h}u_{h}\right\Vert _{L^{2}(X)}=O(h^{N}).\]

\end{defn}
L'ensemble $\mathcal{A}_{h}(X)$ est un $\mathbb{C}_{h}$-module pour
les lois usuelles des fonctions. Un premier fait important est que
par le théorème de Calderon-Vaillancourt, on a l'inclusion : $L^{2}(X)\subset\mathcal{A}_{h}(X).$
\begin{example}
Les fonctions WKB%
\footnote{Pour Wentzel, Kramers et Brillouin.%
} de la forme : \[
u_{h}(x)={\displaystyle \alpha(x)e^{i\frac{S(x)}{h}}}\]
$S$ étant une fonction réelle $\mathcal{C}^{\infty}$, sont des fonctions
admissibles stables par l'action d'un opérateur pseudo-différentiel.
\end{example}

\subsection{Micro-support et microfonctions}

Lorsque on regarde une équation du type $P_{h}u_{h}=0$, l'étude dans
$T^{*}X$ de la fonction symbole $p$ de l'opérateur pseudo-différentiel
$P_{h}$ permet de localiser les singularités de la solution $u_{h}$
grâce a la notion de micro-support. Historiquement la notion de micro-support
a été introduite par Sato et Hörmander. On va donner une définition
proche de celle de Hörmander : a tout élément $u_{h}$ du $\mathbb{C}_{h}$-module
des fonctions admissibles est associé un sous-ensemble de $T^{*}X$,
cet ensemble, nommé micro-support%
\footnote{Ou front d'ondes.%
} décrit la localisation de la fonction $u_{h}$ dans l'espace des
phases.
\begin{defn}
Soit $u_{h}\in\mathcal{A}_{h}(X)$, on dira que $u_{h}$ est négligeable
au point $m\in T^{*}X$, si et seulement s'il existe $P_{h}$ un opérateur
pseudo-différentiel elliptique en $m$ tels que :\[
\left\Vert P_{h}u_{h}\right\Vert _{L^{2}(X)}=O(h^{\infty}).\]
On définit alors $MS(u_{h})$, le micro-support de $u_{h}$ comme
le complémentaire dans $T^{*}X$ de l'ensemble des points $m\in T^{*}X$
où $u_{h}$ est négligeable.
\end{defn}
Moralement le micro-support de $u_{h}$ est le complémentaire de l'ensemble
des directions où, a une variante prés, la transformée de Fourier
de $u_{h}$ est à décroissance rapide. Parmi les propriétés liées
au micro-support nous avons que si $P_{h}$ est un opérateur pseudo-différentiel
de symbole principal $p$ alors on a l'implication : \[
P_{h}u_{h}=O(h^{\infty})\Rightarrow MS(u_{h})\subset p^{-1}(0).\]
Donc si par exemple $P_{h}$ est un opérateur de symbole principal
$p$, $\lambda$ un scalaire, et si $u_{h}$ est une fonction non
nulle telle que : $\left(P_{h}-\lambda I_{d}\right)u_{h}=O(h^{\infty})$
alors $MS(u_{h})\subset p^{-1}(\lambda).$ Ceci est une propriété
fondamentale de l'analyse microlocale: elle donne une localisation
des fonctions propres dans l'espace des phases.
\begin{example}
Pour une fonction WKB : $u_{h}(x)={\displaystyle {\displaystyle \alpha(x)e^{i\frac{S(x)}{h}}}}$
on a : \[
MS(u_{h})=\left\{ \left(x,dS(x)\right),\alpha(x)\neq0\right\} .\]
\end{example}
\begin{defn}
Soient $u_{h},v_{h}\in\mathcal{A}_{h}(X)^{2}$, on dira que $u_{h}=v_{h}+O(h^{\infty})$
sur un ouvert $U\subset T^{*}X$ si et seulement si :\[
MS(u_{h}-v_{h})\cap U=\textrm{Ø}.\]

\end{defn}
Avec les propriétés du micro-support, on peut montrer que pour tout
ouvert $U$ de $T^{*}X$, l'ensemble $\left\{ u_{h}\in\mathcal{A}_{h}(X)/MS(u_{h})\cap U=\textrm{Ø}\right\} $
est un $\mathbb{C}_{h}-$sous-module de $\mathcal{A}_{h}(X)$, on
peut alors définir l'espace des micro-fonctions :
\begin{defn}
Soit $U$ un ouvert non vide de $T^{*}X$, on définit l'espace des
micro-fonctions sur $U$ comme étant le $\mathbb{C}_{h}-$module quotient:\[
\mathcal{M}_{h}(U):=\mathcal{A}_{h}(X)/\left\{ u_{h}\in\mathcal{A}_{h}(X),\, MS(u_{h})\cap U=\textrm{Ø}\right\} .\]

\end{defn}
Les opérateurs pseudo-différentiels agissent sur $\mathcal{M}_{h}(U)$,
en effet : pour tout opérateur pseudo-différentiel $P_{h}$ on a :
\[
MS(P_{h}u_{h})\subset MS(u_{h})\]
et ainsi $P_{h}\left(\mathcal{M}_{h}(U)\right)\subset\mathcal{M}_{h}(U).$

\subsection{Analyse microlocale et faisceaux}

Le langage le plus adapté à l'analyse microlocale est le langage des
faisceaux. Sans rentrer dans les détails, en voici le principe : a
tout triplet $\left(P_{h},\lambda,U\right)$ où $P_{h}$ est un opérateur
pseudo-différentiel, $\lambda$ un scalaire de l'anneau $\mathbb{\mathbb{C}}_{h}$
et $U$ un ouvert non vide de $T^{*}X$, on peut associer l'ensemble
$\mathcal{L}\left(P_{h},\lambda,U\right)$ des microfonctions $u_{h}$
solutions dans l'ouvert $U$ de $\left(P_{h}-\lambda I_{d}\right)u_{h}=O(h^{\infty})$.
L'ensemble $\mathcal{L}\left(P_{h},\lambda,U\right)$ est un $\mathbb{C}_{h}$-module,
et si $\Omega$ désigne un ensemble d'indices quelconque, la famille
d'ensembles $\left\{ \mathcal{L}\left(P_{h},\lambda,U_{x}\right),\, x\in\Omega\right\} $
est un faisceau au dessus de ${\displaystyle \bigcup_{x\in\Omega}}U_{x}$.
En effet toute solution peut être restreinte sur des ouverts plus
petits d'une unique manière, et deux solutions $u_{h}$ définie sur
un ouvert $U_{x}$ et $v_{h}$ définie sur un autre ouvert $U_{y}$
et telles que $u_{h}=v_{h}$ sur l'ouvert $U_{x}\cap U_{y}$ peuvent
être misent ensemble pour former une solution globale sur l'ouvert
$U_{x}\cup U_{y}$. Ce faisceau est supporté%
\footnote{Au sens du micro-support.%
} sur l'ensemble $p^{-1}(\lambda)\subset T^{*}X$.

\subsection{Théorème d'Egorov et opérateurs intégraux de Fourier}

Pour finir donnons le théorème d'Egorov qui permet de définir rapidement
la notion d'opérateur intégral de Fourier, voir par exemple \textbf{{[}57{]},
{[}46{]}} :
\begin{thm}
\textbf{(Egorov)} : Soient $\left(T^{*}X,d\alpha\right)$ et $\left(T^{*}Y,d\beta\right)$
deux variétés symplectomorphe : il existe $\chi$ un symplectomorphisme
de $T^{*}X$ dans $T^{*}Y$. On supposera que $\chi$ est exact :
$\chi^{*}\beta-\alpha$ est une 1-forme exacte sur $X$. Alors il
existe $\widetilde{\chi}$ un morphisme de $\mathbb{C}_{h}$-module
de $\mathcal{M}_{h}(X)$ dans $\mathcal{M}_{h}(Y)$ inversible tel
que pour tout $a\in\mathcal{M}_{h}(Y)$, en notant par $\hat{a}=\mathbf{O}_{p}^{w}(a)$,
l'opérateur :\[
B=\widetilde{\chi}^{-1}\circ\hat{a}\circ\widetilde{\chi}\]
est un opérateur pseudo-différentiel sur $\mathcal{M}_{h}(X)$, et
dont le symbole principal est donné par $a_{0}\circ\chi$, $a_{0}$
étant le symbole principal de $\hat{a}$. On dit que $\widetilde{\chi}$
est un opérateur intégral de Fourier associé à $\chi$.
\end{thm}
En fait il y a toute une théorie sur les opérateurs intégraux de Fourier
(voir \textbf{{[}53{]}} et {[}\textbf{54{]}}), ces derniers on en
effet une représentation intégrale. Leurs noyaux sont cependant plus
compliqués et plus délicats à manipuler que les noyaux des opérateurs
pseudo-différentiels.

\section{Vers la quantification par déformation}

Avec la théorie des opérateurs pseudo-différentiels, on dispose d'une
quantification sur $\mathbb{R}^{2n}$ ou sur $T^{*}X$ . Cette quantification
n'est en fait pas idéale au sens où l'axiome \textbf{(ii)} est vérifié
avec un reste en $O(h^{2}).$ Donc si accepte de laisser tendre $h$
vers $0$ on se rappoche d'une quantification idéale.

\subsection{Cas de $T^{*}X$ }

Sur la variété symplectique $\mathbb{R}^{2n}$ ou sur $T^{*}X$ la
théorie des opérateurs pseudo-différentiels fournit une quantification
(tpas idéale mais presque). De manière concrète pour toute fonction
symbole $a$ {}``convenable'' on l'existence de l'opérateur linéaire
$\hat{a}$ définit par la quantification de Weyl :\[
\hat{a}:=\mathbf{O}_{p}^{w}(a).\]
Dans l'exemple 6.2 on a vu que les axiomes \textbf{(i)} et \textbf{(iii)}
de la quantification sont vérifiés; l'axiome \textbf{(iv) }est lui
aussi vrai. Par contre l'axiome\textbf{ (ii) }pose problème\textbf{
:} généralement pour tout couple de symboles $\varphi,\psi$ on a
que : \[
\mathbf{O}_{p}^{w}(\varphi)\circ\mathbf{O}_{p}^{w}(\psi)\neq\mathbf{O}_{p}^{w}(\varphi\psi)\]
mais le théorème 6.4 assure l'existence d'un unique symbole $\pi$
donné par la formule de Moyal :\[
\pi=\sum_{j=0}^{\infty}\frac{h^{j}}{j!(2i)^{j}}\varphi\left(\sum_{p=1}^{n}\overleftarrow{\partial_{\xi_{p}}}\overrightarrow{\partial_{x_{p}}}-\overleftarrow{\partial_{x_{p}}}\overrightarrow{\partial_{\xi_{p}}}\right)^{j}\psi\]
vérifiant la relation :\[
\mathbf{O}_{p}^{w}(\varphi)\circ\mathbf{O}_{p}^{w}(\psi)=\mathbf{O}_{p}^{w}(\pi).\]
On définit alors le produit de Moyal des symboles\textit{ }$\varphi$
et $\psi$ par :\[
(\varphi,\psi)\mapsto\varphi\star\psi:=\pi\]
soit encore : \[
\widehat{\varphi\star\psi}=\widehat{\varphi}\widehat{\psi}.\]
Et comme d'aprés la formule de Moyal : $\varphi\star\psi:=\varphi\psi+O(h)$,
on en déduit alors la formule :\[
\widehat{\varphi\star\psi}=\widehat{\varphi\psi}+O(h).\]
Par conséquent si on calcule le commutateur de deux symboles quantifiés
\foreignlanguage{english}{$\widehat{\varphi}$} et $\widehat{\psi}$
on obtient :\[
\left[\widehat{\varphi},\widehat{\psi}\right]=\widehat{\varphi}\widehat{\psi}-\widehat{\psi}\widehat{\varphi}\]
\[
=\widehat{\varphi\star\psi}-\widehat{\psi\star\varphi}\]
\[
=\widehat{\varphi\star\psi-\psi\star\varphi}\]
et comme par la formule de Moyal $\varphi\star\psi-\psi\star\varphi=\frac{h}{i}\left\{ \varphi,\psi\right\} +O(h^{2})$
on arrive la formule :\[
\left[\widehat{\varphi},\widehat{\psi}\right]=\frac{h}{i}\widehat{\left\{ \varphi,\psi\right\} }+O(h^{2}).\]
On vient donc de montrer que l'axiome\textbf{ (ii) }est vrai modulo
$O(h^{2})$. 

En fait, on peut voir le produit non-commutatif $\star$ comme définit
sur un anneau de fonctions en oubliant la représentation par les opérateurs
pseudo-différentiels. La quantification apparaît donc comme une déformation
du commutatif vers le non-commutatif, la trace de la non-commutativité
étant donné par le crochet $\left\{ .,.\right\} $.

\subsection{Cas général}

Le programme de quantification par déformation est initié par Bayer,
Flato,... il consiste à essayer de définir sur chaque variété symplectique
un produit $\star$ analogue à celui de Moyal. Localement cela est
facile (on se ramène à $\mathbb{R}^{2n}$ par le théorème de Darboux),
mais bien évidement la difficulté de ce programme est le passage du
local au global. Concrètement quand on dispose d'un produit $\star$
sur une variété $M$, on peut quantifier : on associe à chaque fonction
$u\in\mathcal{C}^{\infty}(M)$ un opérateur $\widehat{u}$ définit
par :\[
\widehat{u}(v):=u\star v.\]
Le problème des produits $\star$ sur des variétés est d'abord celui
de leurs existences et de leurs classifications. Parmis les résultats
majeurs citons d'abord les travaux de De Wilde-Lecompte\textbf{ {[}51{]}}
qui en 1983 montrent que toute variété symplectique admet un produit
$\star$. En 1997, Kontesvitch\textbf{ {[}88{]} }montre que toute
variété de Poisson%
\footnote{C'est une variété avec une structure de poisson sur l' algèbre des
fonctions.%
} admet un produit $\star$.

\section{Les systèmes complètement intégrables }

Pour finir ce chapitre on va donner les principaux résultats connus
sur les systèmes complètement intégrables symplectiques et semi-classique.

\subsection{Définitions}

En physique hamiltonienne un système intégrable est un système possédant
un nombre suffisant de constantes de mouvement indépendantes; mathématiquement
on a la définition suivante :
\begin{defn}
Un système intégrable classique est la donnée d'une variété symplectique
$\left(M,\omega\right)$ de dimension $2n$ et de $n$ fonctions $\left(f_{1},...,f_{n}\right)$
de l'algèbre $\mathcal{C}^{\infty}(M)$ telles que les différentielles
$\left(df_{i}(x)\right)_{i=1,...,n}$ sont libres presque partout
sur $M$; et telles que pour tout indices $i,j$ on ait $\left\{ f_{i},f_{j}\right\} =0$.
On définit alors l'application moment classique associée :\[
\mathbf{f}\,:\left\{ \begin{array}{cc}
M\rightarrow\mathbb{R}^{n}\\
\\x\mapsto\left(f_{1}(x),...,f_{n}(x)\right).\end{array}\right.\]
\end{defn}
\begin{example}
L'oscillateur harmonique classique : sur la variété symplectique $\mathbb{R}^{2}$
on note par $p$ l'oscillateur harmonique : $\mathbf{p}(x,\xi)=p(x,\xi)=\left(x^{2}+\xi^{2}\right)/2$.
En effet $dp=\left(\begin{array}{c}
x\\
\xi\end{array}\right)$ est non nul presque partout, donc libre presque partout. \end{example}
\begin{defn}
Un système complètement intégrable semi-classique sur une variété
$X$ est la donné de $n$ opérateurs pseudo-différentiels $P_{1},...,P_{n}$
sur $L^{2}(X)$ tels que pour tout indices $i$ et $j$ on ait $\left[P_{i},P_{j}\right]=0$
et dont les symboles principaux forment un système complètement intégrable
symplectique sur $M:=T^{*}X$. On notera par $\mathbf{P}:=(P_{1},...,P_{n})$
l'application moment quantique et par $\mathbf{p}:=(p_{1},..,.p_{n})$
l'application moment classique associée aux symboles principaux de
$\mathbf{P}.$ 
\end{defn}
En géométrie symplectique et en anaylse semi-classique, l'application
moment joue un rôle important pour classifier les systèmes complètement
intégrables. En analyse semi-classique on travaille avec le spectre
(spectre exact et semi-classique) d'un système complètement intégrable
:
\begin{defn}
Le spectre exact $\sigma(P_{1},...,P_{n})$ d'un système complètement
intégrable semi-classique $(X,\mathbf{P})$, ou encore spectre conjoint
exact des opérateurs pseudo-différentiels $P_{1},...,P_{n}$, est
définit par :\[
\sigma(P_{1},...,P_{n})\]
\[
:=\left\{ \left(\lambda_{1},\ldots,\lambda_{n}\right)\in\mathbb{R}^{n},\,\exists u\in L^{2}(\mathbb{R}^{n}),\, u\neq0;P_{j}u=\lambda_{j}u\right\} .\]

\end{defn}
La définition semi-classique est :
\begin{defn}
Le spectre semi-classique $\Sigma_{h}(P_{1},...,P_{n})$ d'un système
complètement intégrable semi-classique $(X,\mathbf{P})$, ou encore
spectre conjoint des opérateurs pseudo-différentiels $P_{1},...,P_{n}$
est définit par :\[
\Sigma_{h}(P_{1},...,P_{n})\]
\[
:=\left\{ \left(\lambda_{1}(h),\ldots,\lambda_{n}(h)\right)\in\mathbb{R}^{n},\,\exists u_{h}\in L^{2}(\mathbb{R}^{n}),\, u_{h}\neq0;\,\left(P_{j}-\lambda_{j}(h)I_{d}\right)u_{h}=O(h^{\infty})\right\} \]
$I_{d}$ étant l'opérateur identité. 
\end{defn}
On appelle multiplicité microlocale de $E_{h}$ la dimension du $\mathbb{C}_{h}$-module
des solutions microlocales de cette équation.

Moralement le spectre semi-classique (ou microlocal) correspond aux
valeurs propres approchées avec une précision d'ordre $O(h^{\infty})$.
Le lien précis entre spectre exact et semi-classique est donné par
la \textbf{{[}119{]}} :
\begin{prop}
Sur un compact $K$ de $\mathbb{R}$, le spectre semi-classique conjoint
$\Sigma_{h}(P_{1},...,P_{n})$ et le spectre exact $\sigma(P_{1},...,P_{n})$
sont liés par \textbf{:\[
\Sigma_{h}(P_{1},...,P_{n})=\sigma(P_{1},...,P_{n})\cap K+O(h^{\infty})\]
}au sens où si $\lambda_{h}\in\Sigma_{h}(P_{1},...,P_{n})$ alors
il existe $\mu_{h}\in\sigma(P_{1},...,P_{n})\cap K$ tel que $\lambda_{h}=\mu_{h}+O(h^{\infty})$;
et si $\mu_{h}\in\sigma(P_{1},...,P_{n})\cap K$ alors il existe $\lambda_{h}\in\Sigma_{h}(P_{1},...,P_{n})$
tel que $\mu_{h}=\lambda_{h}+O(h^{\infty}).$ De plus pour toute famille
$\lambda_{h}$ ayant une limite finie $\lambda\in K$ lorsque $h\rightarrow0$,
si la multiplicité microlocale de $\lambda_{h}$ est bien définie
et est finie, alors elle est égale pour $h$ assez petit au rang du
projecteur spectral conjoint des $P_{j}$ sur une boule de diamètre
$O(h^{\infty})$ centrée autour de $\lambda_{h}$.\end{prop}
\begin{example}
L'oscillateur harmonique quantique : ici on prend $M=\mathbb{R}$
et :\textit{\[
\mathbf{P}=P=-\frac{h^{2}}{2}\frac{d^{2}}{dx^{2}}+\frac{x^{2}}{2}\]
}alors l'application moment classique est $\mathbf{p}(x,\xi)=p(x,\xi)=\left(x^{2}+\xi^{2}\right)/2$.
Dans cet exemple, le spectre exact et semi-classique coincide et\textit{
$\sigma(P)=\left\{ h\left(n+\frac{1}{2}\right),\, n\in\mathbb{N}\right\} $.}
\end{example}

\subsection{Théorie locale}

\subsubsection*{Des points réguliers ...}

Les points réguliers de l'application moment classique sont les points
$m\in M$ tels que les différentielles $\left(dp_{i}(m)\right)_{i=1,...,n}$
sont libres. Les points réguliers d'un système complètement intégrable
ont une description symplectique locale simple donnée par le théorème
de Darboux-Carathéodory (voir par exemple \textbf{{[}9{]}}) :
\begin{thm}
\textbf{(Darboux-Carathéodory). }Soit $\left(M,\omega,\mathbf{p}=(p_{1},...,p_{n})\right)$
un système complètement intégrable de dimension $2n.$ Pour tout point
$m\in M$ régulier, ie : $\left(dp_{i}(m)\right)_{i=1,...,n}$ est
une famille libre; il existe un système de coordonnées canonique locales
$\left(x_{1},...,x_{n},\zeta_{1},...,\zeta_{n}\right)$ tel que sur
un voisinage de $m$ :\[
\zeta_{j}=p_{j}-p_{j}(m).\]

\end{thm}
Ce théorème a un analogue semi-classique du à Y. Colin de Verdière
en 1979 \textbf{{[}36{]}} :
\begin{thm}
Soit $\left(M,\omega,F=(f_{1},...,f_{n})\right)$ un système complètement
intégrable de dimension $2n.$ Pour tout point $m\in M$ régulier,
il existe $U$ un opérateur intégral de Fourier unitaire définie microlocalement
prés de $m$ tel que sur un voisinage microlocal de $m$ :\[
U\left(P_{j}-p_{j}(m)I_{d}\right)U^{-1}=\frac{h}{i}\frac{\partial}{\partial x_{j}}.\]

\end{thm}
Le théorème de Darboux-Carathéodory semi-classique permet de faire
une description précise de l'ensemble des micro-solutions des équations
$P_{j}u_{h}=O(h^{\infty})$. Les résultats d'analyse microlocale nous
informe déjà que les solutions $u_{h}$ sont localisées sur ${\displaystyle \bigcap_{j=1}^{n}}p_{j}^{-1}(0)$;
mais en fait on a bien mieux :
\begin{thm}
Pour tout point $m\in M$ régulier de $\mathbf{p}=(p_{1},...p_{n})$
tels que $\mathbf{p}(m)=0$, le faisceau des micro-solutions de l'équation\[
P_{j}u_{h}=O(h^{\infty})\]
est un faisceau en $\mathbb{C}_{h}$ module libre de rang 1 engendré
par $U^{-1}(\mathbf{1})$, où $U$ est donné par le précédent théorème
et $\mathbf{1}$ est une micro-fonction égale à $1$ près de l'origine. 
\end{thm}

\subsubsection*{... aux points singuliers non-dégénérés}

Il y a toute une théorie sur l'étude et la classification des singularités
des applications moment. Le cas de certaines singularités générique,
dites singularités non-dégénérés est bien connu. Suivant une classification
algébrique dut à Williamson \textbf{{[}120{]}} datant de 1936, ces
singularités sont de trois types :
\begin{enumerate}
\item les singularités elliptiques : $q_{i}=x_{i}\zeta_{i};$
\item les singularités hyperboliques $q_{i}=(x_{i}^{2}+\zeta_{i}^{2})/2;$
\item et les singularités loxodromique (ou focus-focus) $q_{i}=x_{i}\zeta_{i+1}-x_{i+1}\zeta_{i}\,\,\textrm{et }\,\, q_{i+1}=x_{i}\zeta_{i}+x_{i+1}\zeta_{i+1}.$
\end{enumerate}
D'après un théorème d'Eliasson de 1984 \textbf{{[}55}{]}, \textbf{{[}56{]}}
toutes ces singularités non-dégénérés sont linéarisable : il existe
$\chi$ un symplectomorphisme local de $\mathbb{R}^{2n}$ dans $M$
tel que : $p_{i}\circ\chi=g_{i}(q_{1},q_{2},\ldots,q_{n}).$ En réalité
Y. Colin de Verdière et J. Vey ont traités le cas de la dimension
1 de ce théorème en 1979 \textbf{{[}35{]}}. Par quantification de
Weyl on obtient l'analogue semi-classique du théorème d'Eliasson (voir
le livre de S. Vu Ngoc\textbf{ }{[}\textbf{119{]}} pour un énoncé
précis).

\subsection{Théorie semi-globale}

\subsubsection*{Des fibres régulières ...}

Une fibre $\Lambda_{c}:=\mathbf{p}^{-1}(c)$ est régulière si et seulement
si tous les points de $\Lambda_{c}$ sont réguliers pour $\mathbf{p}.$
Les fibres régulières sont décrites par le théorème actions-angles,
nommé aussi théorème d'Arnold-Liouville-Mineur, qui donne la dynamique
classique au voisinage d'une fibre régulière connexe et compacte :
le flot hamiltonien associé à une intégrale première est quasi-périodique
(droite s'enroulant sur un tore) :
\begin{thm}
\textbf{(Actions-angles).} Soit $\left(M,\omega,\mathbf{p}=(p_{1},...,p_{n})\right)$
un système complètement intégrable de dimension $2n.$ Soit $\Lambda_{c}$
une composante connexe compacte de $\mathbf{p}^{-1}(c)$ telle que
tous les points de la fibre $\Lambda_{c}$ sont réguliers; alors toutes
les fibres dans un voisinage de $\Lambda_{c}$ sont des tores et il
existe $\varphi$ un symplectomorphisme :\textup{\[
\varphi\,:\, T^{*}\mathbb{T}^{n}\rightarrow M\]
}qui envoi la section nulle $\mathbb{T}^{n}\times\left\{ 0\right\} $
de $T^{*}\mathbb{T}^{n}$ sur $\Lambda_{c}$ et tel que :\textup{\[
\mathbf{p}\circ\varphi=\chi\left(\zeta_{1},...,\zeta_{n}\right)\]
}où $\chi$ est un difféomorphisme local de $\mathbb{R}^{2n}$ laissant
fixe l'origine. 
\end{thm}
On dit que les variables $(x_{1},\ldots,x_{n})$ sont les variables
actions et $(\zeta_{1},...,\zeta_{n})$ les variables angles. San
Vu Ngoc a donné la version semi-classique de ce théorème \textbf{{[}116{]}},\textbf{{[}117{]}
}:
\begin{thm}
\textbf{(Actions-angles semi-classique).} Si la fibre $\Lambda_{c}$
est régulière il existe $U$ un opérateur intégral de Fourier associé
à $\chi$ un symplectomorphisme exact :\textup{\[
\chi\,:\, T^{*}\mathbb{T}^{n}\rightarrow M\]
}qui envoi la section $\zeta=a$ sur $\Lambda_{c}$ et il existe aussi
des séries formelle $\lambda_{j}(h)\in$$\mathbb{C}[[h]]$, $j=1,...,n$
telle que microlocalement au voisinage de la section $\zeta=a$ on
ait \[
U\left(P_{1}-p_{1}(m)I_{d},...,P_{n}-p_{n}(m)I_{d}\right)U^{-1}=N\left(\frac{h}{i}\frac{\partial}{\partial x_{1}}-\lambda_{1}(h)I_{d},...,\frac{h}{i}\frac{\partial}{\partial x_{n}}-\lambda_{n}(h)I_{d}\right)\]
où $m$ est un point quelconque de la fibre $\Lambda_{c}$ et $N$
est un matrice 2$\times$2 inversible à coefficients dans le $\mathbb{C}_{h}$-module
des opérateurs pseudo-différentiel.
\end{thm}
Le théorème actions-angles semi-classique fournit donc une description
locale du spectre d'un système intégrable; en effet autour d'une valeur
régulière de l'application moment classique $\mathbf{p}$ le spectre
conjoint semi-classique \foreignlanguage{english}{$\Sigma_{h}(P_{1},...,P_{n})$}
est localement difféomorphe au réseau $h\mathbb{Z}^{n}$ (voir \textbf{{[}119{]}}).
Précisément du théorème actions-angles semi-classique on arrive aux
fameuses règles de quantification de Bohr-Sommerfeld :
\begin{thm}
\textbf{(Règles de Bohr-Sommerfeld). }Il existent $n$ fonctions $S_{j}(E)=S_{j}(E,h),\, j=1,...,n$
admettant des développements asymptotiques en puissance de $h$ avec
des coefficients de classe $\mathcal{C}^{\infty}$ par rapport à $E=\left(E_{1},\cdots,E_{n}\right)\in\mathbb{R}^{n}$,
telles que les équations :\[
\left(P_{j}-E_{j}I_{d}\right)u_{h}=O(h^{\infty})\]
admettent une solution microlocale $u_{h}$ avec son microsupport
$MS(u_{h})=\Lambda_{E}$ si et seulement si $S_{j}(E)\in2\pi h$. 
\end{thm}
Décrivons un peu la construction de ces fonctions $S_{j}(E)$. Pour
fixer les idées, plaçons nous dans le cas de la dimension symplectique
deux; donc \textbf{$\mathbf{P}=P$}. D'après un des théorèmes précédent
pour une fibre $\Lambda_{E}:=\mathbf{p}^{-1}(E)$ compacte, connexe
et régulière, toute microsolution $u_{h}$ de $\left(P-EI_{d}\right)u_{h}=O(h^{\infty})$
est engendré par $U^{-1}(\mathbf{1})$. La théorie des opérateurs
intégraux de Fourier montre que $u_{h}$ est nécessairement du type
WKB. On va maintenant décrire comment on prolonge une microsolution
d'un ouvert à un autre le long d'une fibre régulière (pour plus de
détails, voir {[}\textbf{119}{]}). Pour commencer on se donne un recouvrement
fini $\left(U_{\alpha}\right)_{\alpha\in\Omega}$ d'ouverts de la
fibre $\Lambda_{E}$. Pour tout couple d'ouverts non vides $U_{\alpha}$
et $U_{\beta}$ du recouvrement tels que $U_{\alpha}\cap U_{\beta}$
est non vide et connexe; si on considère alors deux microfonctions
$\varphi_{\alpha}$ et $\varphi_{\beta}$ solutions de $\left(P-\lambda I_{d}\right)u_{h}=O(h^{\infty})$
microlocalement sur les ouverts respectifs $U_{\alpha}$ et $U_{\beta}$,
les microfonctions $\varphi_{\alpha}$ et $\varphi_{\beta}$ sont
respectivement engendrés par $U^{-1}\left(\mathbf{1}_{\alpha}\right)$
et par $U^{-1}\left(\mathbf{1}_{\beta}\right)$, $\mathbf{1}_{\alpha}$
et $\mathbf{1}_{\beta}$ étant égale à $1$ microlocalement sur $U_{\alpha}$
et respectivement sur $U_{\beta}$. En se plaçant sur $U_{\alpha}\cap U_{\beta}$
et en utilisant l'argument de la dimension $1$ on a l'existence de
$c_{\alpha,\beta}\in\mathbb{C}_{h}$ tel que sur $U_{\alpha}\cap U_{\beta}$
on ait\[
U^{-1}\left(\mathbf{1}_{\alpha}\right)=c_{\alpha,\beta}U^{-1}\left(\mathbf{1}_{\beta}\right)\]
et donc, sur $U_{\alpha}\cap U_{\beta}$ on a : \[
\mathbf{1}_{\alpha}=c_{\alpha,\beta}\mathbf{1}_{\beta}.\]
La théorie des opérateurs intégraux de Fourier montre (voir \textbf{{[}118{]}},\textbf{{[}119{]}})
que la constante $c_{\alpha,\beta}$ s'écrit sous la forme : \[
c_{\alpha,\beta}=e^{\frac{iS_{\alpha\beta}}{h}}\]
le scalaire $S_{\alpha\beta}$ étant dans $\mathbb{C}_{h}$ est dépendant
de la variable $E$. 

Plus généralement pour une famille finie $\left(U_{k}\right)_{k=1,...,l}$
d'ouverts non vides recouvrant une partie compacte et connexe de la
fibre régulière $\Lambda_{E}$ telle que pour tout indice $k\in\left\{ 1,...,l-1\right\} ,$
$U_{k}\cap U_{k+1}$ est non vide et connexe. Sur chaque ouvert $U_{k}$
on a un générateur $\mathbf{1}_{k}$ de $\mathcal{L}\left(P,\lambda,U_{k}\right)$
et pour tout indice $k\in\left\{ 1,...,l-1\right\} $ il existe $c_{k,k+1}=e^{\frac{iS_{k,k+1}}{h}}\in\mathbb{C}_{h}$
tel que :

\[
\mathbf{1}_{k}=c_{k,k+1}\mathbf{1}_{k+1}\]
ainsi nous avons alors l'égalité suivante \[
\mathbf{1}_{1}=c_{1,2}c_{2,3}\ldots c_{l-1,l}\mathbf{1}_{l}.\]

On peut donc écrire $\mathbf{1}_{1}=e^{\frac{iS_{1,l}}{h}}1_{l}$
où on a posé \[
S_{1,l}=\sum_{k=1}^{l-1}S_{k,k+1}.\]
On notera par $S$ sa classe de cohomologie à valeur dans $\mathbb{R}/2\pi h\mathbb{Z}$
associée. La dépendance en la variable $E$ est lisse : les fonctions
$E\mapsto S_{.}(E)$ sont $\mathcal{C}^{\infty}$(voir \textbf{{[}116{]}},\textbf{{[}118{]}}
et \textbf{{[}119{]}}). Ainsi, avoir une solution globale sur la fibre
signifie que le réel $S(E)$ est un multiple de $2\pi h$. 

\begin{center}
\includegraphics[scale=0.5]{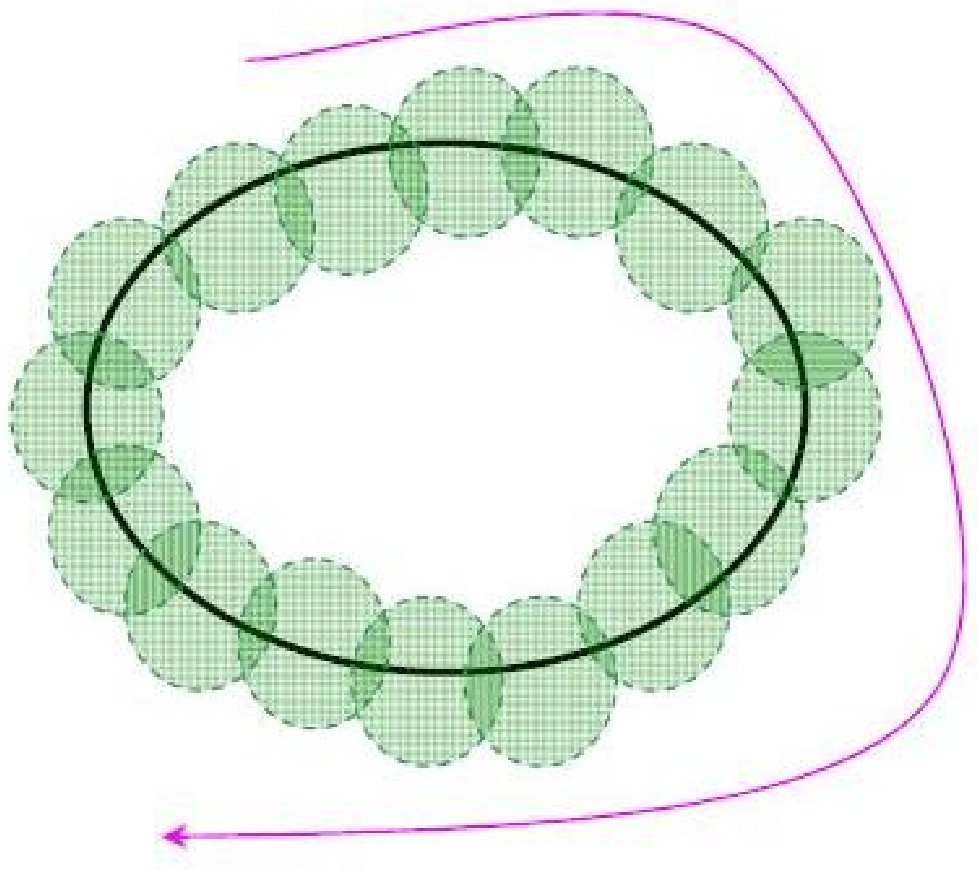}
\par\end{center}

\begin{center}
\textit{Fig. 2: Un recouvrement par des ouverts d'une fibre régulière.}
\par\end{center}

\vspace{0.25cm}

Dans les travaux de B. Hellfer et D. Robert on a un énoncé plus explicite
des règles de quantification de Bohr-Sommerfeld (voir l'article de
B. Hellfer et D. Robert {[}77\textbf{{]}}). Les Hypothèses fonctionnelles
de B. Hellfer et D. Robert {[}\textbf{77{]}} et \textbf{{[}109{]}})
sont les suivantes: 

\textbf{(H1) }Le symbole $p$ de $P_{h}$ est réel et lisse.

\textbf{(H2)} Le symbole principale $p_{0}$ de $P_{h}$ est borné
inférieurerement : il existe $c_{0}>0$ et $\gamma_{0}\in\mathbb{R}$
tels que pour tout $z\in T^{*}\mathbb{R}$, $c_{0}<p_{0}(z)+\gamma_{0}$.
De plus on suppose que \foreignlanguage{english}{$p(z)+\gamma_{0}$}
est un poids tempéré : il existe $C>0$ et $M\in\mathbb{R}$ tels
que pour tout $(z,z^{\prime})\in\left(T^{*}\mathbb{R}\right)^{2}$
on ait \[
p_{0}(z)+\gamma_{0}\leq C\left(p_{0}(z^{\prime})+\gamma_{0}\right)\left(1+\left|z-z^{\prime}\right|\right)^{M}.\]

\textbf{(H3)} Pour tout indice $j\in\mathbb{N}$ et pour tout multiindex
$\alpha$ il existe $c>0$ tel que pour tout $z\in T^{*}\mathbb{R}$
\[
\left|\frac{\partial^{\alpha}H_{j}}{\partial z^{\alpha}}\right|\leq c\left(p_{0}(z)+\gamma_{0}\right).\]

\textbf{(H4)} Il existe $N_{0}\in\mathbb{N}$ tel que pour tout $N\ge N_{0}$
et pour tout $\gamma>0$ il existe $A>0$ tel que pour tout $z\in T^{*}\mathbb{R}$
\[
\left|\frac{\partial^{\alpha}}{\partial z^{\alpha}}\left(p(z)-\sum_{j=0}^{N}p_{j}(z)h^{j}\right)\right|\leq Ah^{N+1}.\]

\begin{rem}
Sous ces 4 hypothèses l'opérateur $P_{h}$ est essentiellement auto-adjoint
\textbf{{[}107{]}}. 
\end{rem}
Les Hypothèses géométriques sont les suivantes: pour un compact $I=\left[E_{-},E_{+}\right]$,
avec $E_{-}$ et $E_{+}$ deux réels.

\textbf{(H5) }L'ensemble $p^{-1}(I)$ est un anneau topologique de
$T^{*}\mathbb{R}$ feuilleté par des trajectoires périodiques du flot
hamiltonien. 

\textbf{(H6)} La fonction $p$ n'a pas de points critiques sur $p^{-1}(I)$. 

Posons aussi $F_{+/-}:=A\left(E_{+/-}\right)$ où $A(x):={\displaystyle \int_{p(y)\leq x}}dy$.
\begin{thm}
(\textbf{{[}HellRober{]}}) Sous les hypothèses précédentes (H1) à
(H6); il existe une fonction $f_{h}$ définie sur $J:=\left[F_{-},F_{+}\right]$
admettant un développement asymptotique en puissance de $h$ avec
des coefficients de classe $\mathcal{C}^{\infty}$ sur $J$ :\[
f_{h}(x)={\displaystyle \sum_{j=0}^{\infty}f_{j}(x)h^{j};}\]
de sorte que les valeurs propres de $P_{h}$ dans le compact $I$
sont données par \[
\lambda_{n}(h)=f_{h}\left(h\left(n+\frac{1}{2}\right)\right)+O(h^{\infty})\]
avec $n\in\mathbb{N}$ tel que $h\left(n+\frac{1}{2}\right)\in J.$
En outre $f_{0}(x)=2\pi A^{-1}(x)$ et $f_{1}\equiv0$.\end{thm}
\begin{rem}
Dans \textbf{{[}22{]}} (voir aussi \textbf{{[}109{]}}) on trouve une
preuve de ce résultat utilisant les états cohérents.
\end{rem}
Avec des hypothèses légèrement différentes Y. Colin de Verdière \textbf{{[}45{]}}
donne le calcul de tous les coefficients dans le développement asymptotique
de la fonction $E\mapsto S(E)$ et donc la valeur des coefficients
$f_{j}.$

\subsection{... aux fibres singulières}

On ne va pas décrire ici toutes la théorie des fibres singulières;
on rappelle juste que dans le cas elliptique : lorsqu'on a une application
moment $\mathbf{p}=(p_{1},...,p_{n})$ avec un point fixe $m$ elliptique,
alors il existe $\chi$ un symplectomorphisme local de $\mathbb{R}^{2n}$
vers $M$ envoyant $0$ sur le point $m$ tel que :\[
\mathbf{p}\circ\chi=\varphi\circ\left((x_{1}^{2}+\zeta_{1}^{2})/2,\ldots,(x_{n}^{2}+\zeta_{n}^{2})/2\right).\]
En version semi-classique, si on est dans le cas de la dimension 2,
$\mathbf{P}=P_{1}$ avec $\mathbf{p}=p_{1}$ ayant une singularité
elliptique en $m$, alors il existe $U$ un opérateur intégral de
Fourier et $\lambda(h)\in$$\mathbb{C}[[h]]$ tel que microlocalement
prés de l'origine on ait :\[
U\left(\mathbf{P}-\mathbf{p}(m)I_{d}\right)U^{-1}=N\left(Q-h\lambda(h)I_{d}\right)\]
où $Q$ est l'oscillateur harmonique de dimension un et $N$ un opétareur
pseudo-différentiel elliptique en $0$. De là on en déduit une condition
de quantification sur la série formelle $\lambda(h)$ et aussi que
l'ensemble des micro-solutions de l'équation :\[
\mathbf{P}u_{h}=O(h^{\infty})\]
est un $\mathbb{C}_{h}$ module libre de rang 1. 

Pour le cas hyperbolique, on va juste parler du ''huit'' en dimension
2. Pour cela, considérons donc un hamiltonien $p\,:\,\mathbb{R}^{2}\rightarrow\mathbb{R}$
tel que $0$ soit valeur critique de $\mathbf{p}=p$ et telle que
la fibre $\mathbf{p}^{-1}(0)$ est compacte et ne contient qu'un unique
point critique $m\in M$ non-dégénéré de type hyperbolique.

\begin{center}
\includegraphics[scale=0.78]{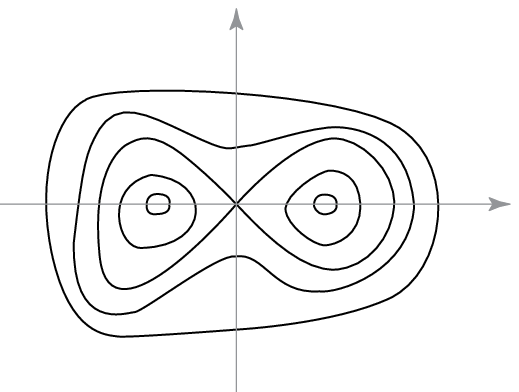}
\par\end{center}

\begin{center}
\textit{Fig. 3: Un {}``huit'' hyperbolique.}
\par\end{center}

\vspace{0.25cm}

La fibre $\Lambda_{0}=\mathbf{p}^{-1}(0)$ est alors un ''huit''
: le feuilletage dans un voisinage de la fibre singulière $\Lambda_{0}$
est difféomorphe à celui du double puits en dimension un (voir {[}\textbf{91{]}}).
Dans le cas de la dimension 2, $\mathbf{P}=P$ avec $\mathbf{p}=p$
ayant une singularité hyperbolique, Y. Colin de Verdière et B. Parisse
on montré que l'ensemble des micro-solutions de l'équation :

\[
\mathbf{P}u_{h}=O(h^{\infty})\]
est un $\mathbb{C}_{h}$ module libre de rang 2. (voir aussi\textbf{
{[}39{]}}, \textbf{{[}40{]}}, \textbf{{[}44{]}} et \textbf{{[}91{]}}).
Dans l'étude des singularités de l'application moment d'un système
complètement intégrable, l'opérateur de Schrödinger avec double puits
est le modèle type pour les singularités non-dégénérées de type hyperbolique.
En effet, pour un hamiltonien $p\,:\, M\rightarrow\mathbb{R}$ tel
que $0$ soit valeur critique de $p$, et tel que les fibres dans
un voisinage de $0$ soient compactes et connexes et ne contiennent
qu'un unique point critique non-dégénéré de type hyperbolique: la
fibre $\Lambda_{0}:=p^{-1}(0)$ est alors un ''huit'' et le feuilletage
dans un voisinage de la fibre singulière $\Lambda_{0}$ est difféomorphe
à celui du double puits. Dans \textbf{{[}91{]} }on montre que le spectre
de l'opérateur $\mathbf{P}$ dans le compact $\left[-\sqrt{h},\sqrt{h}\right]$
est constitué de deux familles de réels $\left(\alpha_{k}(h)\right)_{k}$
et $\left(\beta_{l}(h)\right)_{l}$ vérifiant : \[
\cdots<\beta_{k+1}(h)<\alpha_{k}(h)<\beta_{k}(h)<\alpha_{k-1}(h)<\cdots\]
et, ils existent $C,C^{\prime}$ deux constantes réelles strictement
positives telles que \[
\frac{Ch}{\left|\ln(h)\right|}\leq\left|\alpha_{k+1}(h)-\alpha_{k}(h)\right|\leq\frac{C^{\prime}h}{\left|\ln(h)\right|},\,\frac{Ch}{\left|\ln(h)\right|}\leq\left|\beta_{k+1}(h)-\beta_{k}(h)\right|\leq\frac{C^{\prime}h}{\left|\ln(h)\right|}.\]
Pour le cas loxodromique (qui existe à partir de la dimension symplectique
4) on renvoi au travaux de S. Vu Ngoc \textbf{{[}118{]}}, \textbf{{[}119{]}}.
\selectlanguage{english}%

\hspace{-0.5cm}

\vspace{1cm}

\hspace{-0.5cm}\textbf{\large Olivier Lablée}{\large \par}

\hspace{-0.5cm}Université Grenoble 1- CNRS

\hspace{-0.5cm}Institut Fourier

\hspace{-0.5cm}UFR de Mathématiques

\hspace{-0.5cm}UMR 5582

\hspace{-0.5cm}BP 74 38402 Saint Martin d'Hères 

\hspace{-0.5cm}mail: \textcolor{blue}{lablee@ujf-grenoble.fr}

\hspace{-0.5cm}page web : \textcolor{blue}{http://www-fourier.ujf-grenoble.fr/\textasciitilde{}lablee/}
\end{document}